\documentclass[11pt,oneside,notitlepage]{article}
\usepackage{geometry}

\usepackage[ngerman,english]{babel}
\usepackage{amsmath}
\usepackage{amsrefs}
\usepackage{amssymb}
\usepackage{amsthm}
\usepackage{caption,subcaption}
\usepackage{cite}
\usepackage{color}
\usepackage{enumitem}
\usepackage{framed}
\usepackage{graphicx}
\usepackage{hyperref}

\usepackage{layout}
\usepackage{mathtools}
\usepackage{mathrsfs}
\usepackage{placeins}
\usepackage{pdfpages}

\usepackage{tikz}
\usepackage{txfonts}
\usepackage[normalem]{ulem}

\newtheorem{theorem}{Theorem}
\newtheorem{assumption}[theorem]{Assumption}

\newtheorem{corollary}[theorem]{Corollary}

\newtheorem{lemma}[theorem]{Lemma}
\newtheorem{notation}[theorem]{Notation}
\newtheorem{proposition}[theorem]{Proposition}
\newtheorem{remark}[theorem]{Remark}

\newcommand{\fr}[2]{\frac{{#1}}{{#2}}}

\newcommand{\half}{\frac{1}{2}}
\newcommand{\thfr}{\frac{3}{4}}

\newcommand{\exponent}{\ensuremath{\gamma}}
\newcommand{\exponentt}{\ensuremath{\theta}}

\newcommand*\xbar[1]{%
   \hbox{%
     \vbox{%
       \hrule height 0.5pt 
       \kern0.5ex
       \hbox{%
         \kern-0.2em
         \ensuremath{#1}%
         \kern-0.1em
       }%
     }%
   }%
}
\DeclarePairedDelimiter\abs{\lvert}{\rvert}
\DeclarePairedDelimiter\norm{\lVert}{\rVert}
\newcommand{\dd}{\ensuremath{\mathrm{d}}}
\newcommand{\energy}{\ensuremath{E}}
\newcommand{\energygap}{\ensuremath{\mathcal{E}}}
\newcommand{\E}{\energygap}
\newcommand{\Ez}{{\E}_0}
\newcommand{\dissipation}{\ensuremath{D}}
\newcommand{\D}{\dissipation}
\newcommand{\V}{\ensuremath{V}}
\newcommand{\Vb}{\bar{V}}
\newcommand{\Vt}{\widetilde{\ensuremath{V}}}
\newcommand{\M}{\ensuremath{M}}
\newcommand{\Mt}{\widetilde{\ensuremath{M}}}
\newcommand{\Mb}{\bar{M}}

\newcommand{\step}{\underline}

\newcommand{\co}{e_*}
\newcommand{\T}{\ensuremath{\tau}}
\newcommand{\U}{\ensuremath{w}}

\newcommand{\g}{\ensuremath{\psi}}

\newcommand{\R}{\mathbb{R}}

\definecolor{darkred}{rgb}{0.9,0.1,0.1}

\def\XXint#1#2#3{{\setbox0=\hbox{$#1{#2#3}{\int}$}
\vcenter{\hbox{$#2#3$}}\kern-.5\wd0}}

\numberwithin{equation}{section}   
\numberwithin{theorem}{section}

\begin{document}

\title{
	Optimal $L^1$-type relaxation rates for the Cahn--Hilliard equation on the line}
\author{
	Felix Otto\begin{footnote}{MPI for Mathematics in the Sciences, E-mail address: \href{mailto:felix.otto@mis.mpg.de}{felix.otto@mis.mpg.de}.}\end{footnote},
	Sebastian Scholtes\begin{footnote}{RWTH Aachen University, E-mail address: \href{mailto:sebastian.scholtes@rwth-aachen.de}{sebastian.scholtes@rwth-aachen.de}.}\end{footnote} and
	Maria G. Westdickenberg\begin{footnote}{RWTH Aachen University, E-mail address: \href{mailto:maria@math1.rwth-aachen.de}{maria@math1.rwth-aachen.de}.}\end{footnote}
}
\date{\today}
\maketitle
\begin{abstract} In this paper we derive optimal algebraic-in-time relaxation rates to the kink for the Cahn-Hilliard equation on the line. We assume that the initial data have a finite distance---in terms of either a first moment or the excess mass---to a kink profile and capture the decay rate of the energy and the perturbation. Our tools include Nash-type inequalities, duality arguments, and Schauder estimates.\\


\end{abstract}

\section{Introduction}\label{S:intro}
In this paper we derive relaxation rates in time for the one-dimensional Cahn--Hilliard equation
\begin{align}\label{ch}
\left\{
\begin{aligned}
	u_t&=-\bigl(u_{xx}-G'(u)\bigr)_{xx} && t>0,&& x\in \R,\\
 u&=u_0 &&t=0,&&x\in\R.
\end{aligned}\right.\end{align}
Here $G$ is a double-well potential with nondegenerate absolute minima at $\pm 1$ (cf.\ Assumption \ref{ass:G});  a canonical choice is $G(u)=\frac{1}{4}(1-u^2)^2$.
The Cahn--Hilliard equation was introduced by Cahn and Hilliard \cite{CH} as a phenomenological model for the phase separation of a binary alloy. It has since been widely used and studied in a variety of application areas as well as in  mathematics; the literature is vast but we refer for instance to \cite{CGS,E,ES,F,NCS}.

Equation \eqref{ch} represents the gradient flow with respect to the $\dot{H}^{-1}$ metric of the scalar Ginzburg--Landau energy
\begin{align}
E(u)=\int \frac{1}{2}u_x^2 +G(u)\,\dd x.\label{energy}
\end{align}
(Above and throughout, integrals and norms are over $\R$ unless otherwise noted.) The minimizer of the energy subject to $\pm 1$ boundary conditions at $\pm\infty$ plays a special role in the dynamics and is often referred to as ``the kink.'' We denote by $v$ the kink normalized
so that $v(0)=0$. (For the canonical potential, $v(x)=\tanh(x/\sqrt{2})$.) We denote the energy of $v$ by
\begin{align}\notag
  \co:=E(v)
  \end{align}
  and note for reference below that
  \begin{align}\notag
 -v_{xx}+G'(v)=0\quad\text{and}\quad \int \frac{1}{2}v_{x}^{2}\,\dd x =\int G(v)\,\dd x=\frac{1}{2}\co.
\end{align}
The other kinks consist of translations of the function $v$. At any given time we define the \emph{shifted kink} $v_{c}(x)=v(x-c)$ as an $L^{2}$
projection of the solution $u$ of \eqref{ch} onto the set of minimizers and we will refer to $c$ as \emph{the shift}.
For future reference we remark that $v_{c}$ satisfies the Euler-Lagrange equation
\begin{align}\label{eq:ELvc}
	\int (u-v_{c})v_{cx}\,\dd x=\int uv_{cx}\,\dd x=0.
\end{align}
We neither need nor assume uniqueness of the shifted kink.
We will often work with
\begin{align}
	f:=u-v,\qquad
	f_{c}:=u-v_{c},\quad\text{and}\quad F_c(x):=\begin{cases}\;\int_{-\infty}^x f_c(y)\,\dd y & x< c\\
	-\int_x^\infty f_c(y)\,\dd y &x>c,\end{cases}\label{fc}
\end{align}
as well as
\begin{align*}
   &\energygap := \energy(u) - \energy(v)=\energy(u)-\co&& \text{(the energy gap)}\\
  \text{and }\qquad &\dissipation := \int\Bigl(\bigl(G'(u) - u_{xx}\bigl)_x\Bigr)^2\dd x&& \text{(the dissipation).}
\end{align*}
We will make use of the gradient flow structure of \eqref{ch} only in the form
\begin{align}
\dot{\E}=-D\quad
\text{and}\quad\norm{u(t,\cdot)-u(s,\cdot)}_{\dot{H}^{-1}}\leq\int_s^t D^\frac{1}{2}\, \dd \tau.\label{diff}
\end{align}
Furthermore we will denote the initial energy gap by $ \Ez:=\E(u_0)$.

In \cite{OW} a relaxation framework was introduced that combined algebraic and differential information on distance, energy gap, and dissipation in order to deduce the rate of relaxation to equilibrium. Specifically, it was shown that for initial disturbances such that
\begin{align*}
	H_{0}:= \norm{u_0-v}_{\dot H^{-1}}^2<\infty
\end{align*}
and if there exists $\epsilon>0$ such that
\begin{align*}
	\energygap_{0}=\E(u_0)\leq 2\co-\epsilon,
\end{align*}
then the energy gap and shift decay with rates
\begin{align}
  \E\lesssim \frac{H_0+\E_0}{t},\qquad c^2\lesssim   \frac{H_0+\E_0}{t^{\frac{1}{2}}}.\label{hed}
\end{align}
For an explanation of the $\lesssim$ notation see Notation \ref{not:lesssim}.
We remark that the information $H_{0}<\infty$ specifies that the longtime limit of the shift is zero.

In this paper we are interested in two different measures of distance from the kink. We will establish decay using and in terms of the quantities
\begin{align}
	 \M&:= \int \abs{F_{c}}\,\dd x &&\text{(the ``first moment'')}\label{eq:definitionM}
	 \intertext{or}
	\V&:= \int \abs{f_{c}}\,\dd x&&\text{(the ``excess mass'').}\label{eq:definitionV}
\end{align}
We are particularly interested in disturbances that
are controlled in $L^1$ but potentially far from the origin---and hence large in $\dot H^{-1}$.
Our result based on $V$ obtains relaxation rates that are---unlike the result
\cite{OW} based on $H_0$---independent of the distance of the disturbance from zero.

Our interest in $L^1$ disturbances is driven in part by coarsening problems. While the relaxation framework of \cite{OW} can be applied to coarsening, cf.\ \cite{SW}, it does not seem to be sufficient to
``go through'' coarsening events, since  the $\dot{H}^{-1}$ distance to the slow manifold need not be order one after the collision of two layers. The excess mass $\V$ seems to be a more natural distance for controlling the dynamics up to and beyond collision events. At the same time, the energy condition $\energygap_{0}\leq 2\co-\epsilon$ is natural
for coarsening: If one considers a solution with four kinks, then as the two nearest kinks collide, the energy is driven below $4e_*$ (and hence the energy gap below $2e_*$).

Notice that $M$ is weaker than the true first moment; indeed, it is easy to check that
\begin{align*}
  M\leq \int|x-c||f_c|\,\dd x,
\end{align*}
but it will be strong enough for our purposes. We define the initial values
\begin{align*}
	M_{0}:=M\bigl(u_0\bigr)\qquad\text{and}\qquad
	V_{0}:=V\bigl(u_0\bigr).
\end{align*}
Our main result is the following theorem.
\begin{theorem}\label{t:mainresult}
	Consider a solution $u$ of the Cahn-Hilliard equation \eqref{ch} where the potential $G$ satisfies Assumption~\ref{ass:G} below.
Suppose that there exists $\epsilon>0$ such that the initial data satisfy
	\begin{align}\label{eq:epsilon}
		\energygap_{0}\leq 2\co-\epsilon.
	\end{align}
	Then the following holds true.
	\begin{itemize}
		\item[(i)]
			If $M_{0}<\infty$ and $\int (u_0-v)\,\dd x=0$, then for all future times
\begin{align}\label{slowdown1}
M\lesssim M_0+1
\end{align}
and the disturbance decays according to
			\begin{align}\label{eq:energydecayM}
				\E\lesssim \min\left\{\E_0,\frac{M_0^2+1}{t^{\frac{3}{2}}}\right\}\quad\text{and}\quad
				 c^2\lesssim\min\left\{M_0^\frac{2}{3}+1,\frac{M_0^2+1}{t}\right\}.
			\end{align}
		\item[(ii)]
			If $V_{0}<\infty$ and $\int (u_0-v)\,\dd x=0$, then for all future times
\begin{align}\label{slowdown2}
V\lesssim V_0+1
\end{align}
and the disturbance is controlled by
			\begin{align}\label{eq:energydecayV}
				\E\lesssim \min\left\{\E_{0},\frac{\V_{0}^{2}+1}{t^{\frac{1}{2}}}\right\}\quad\text{and}\quad
				c^2\lesssim \V_{0}^2+1.
			\end{align}
	\end{itemize}
\end{theorem}
\begin{remark}
The proof in  case (ii) is more involved than that in case (i), since it requires establishing the same decay estimates for the dual equation in the case of a moving boundary as on a fixed domain; cf. Subsection \ref{subsection:parabolicvariable}. While the proof is more complicated, the result is  more surprising: We see that $L^1$ perturbations---which, unlike $M$ or the $\dot{H}^{-1}$ distance considered in \cite{OW} are independent of the distance from the origin---remain bounded and that this suffices to deduce $t^\frac{1}{2}$ energetic decay and boundedness of the shift $c$.
\end{remark}
\begin{remark}
In case (i) we obtain the same decay as for diffusion on the half-line with vanishing boundary data.
This behavior is natural because
on large scales, the kink acts as a sink.
We refer the reader to the heuristics in \cite[Subsection 1.3]{OW} for a formal justification of this remark and a discussion of optimality.
See also the remarks in Subsection \ref{ss:previous}, below.
\end{remark}
\begin{remark}
  One should view the condition $\int (u_0-v)\,\dd x=0$ as a normalization. If instead
  \begin{align*}
    \int (u_0-v(\cdot-c_*))\dd x=0
      \end{align*}
 for some $c_*\neq 0$, then decay is obtained for $\abs{c(t)-c_*}$.
\end{remark}

Using the differential inequality
\begin{align}
  \frac{dD}{dt}\lesssim D^\frac{3}{2}\label{dode}
\end{align}
(cf. \cite[(1.24)]{OW}) and arguing as in the proof of \cite[(1.15)]{OW}, we obtain the following corollary.
\begin{corollary}
  In case (i), the dissipation satisfies the pointwise bound
  \begin{align*}
    D\lesssim \frac{M_0^2+1}{t^\frac{5}{2}}.
  \end{align*}
\end{corollary}
Using \eqref{linf} below, the energy and dissipation bounds yield the following $L^\infty$-estimate.
\begin{corollary}
  In case (i), the solution converges to the $L^2$-closest kink with rate
  \begin{align}
    \norm{f_c}_\infty\lesssim \frac{M_0+1}{t}.\label{infdecay}
  \end{align}
\end{corollary}
\begin{remark}[Scale separation in case (i)]
As in \cite{OW}, we obtain in case (i) a faster relaxation (at rate $t^{-1}$ in $L^\infty$) to \textbf{some kink} via \eqref{infdecay} followed by a slower relaxation (at rate $t^{-\frac{1}{2}}$ in $L^\infty$) to \textbf{the centered kink} via the triangle inequality and \eqref{eq:energydecayM}.
\end{remark}
\begin{remark}[$L^\infty$ bound and scale separation in case (ii)]
We can use \eqref{dode} to deduce an $L^\infty$ bound in case (ii) as well, but the rate does not appear to be optimal. The result is
\begin{align*}
      D\lesssim \max\left\{\frac{V_0^2+1}{t^\frac{3}{2}}, \frac{V_0^4+1}{t}\right\}
\end{align*}
and
\begin{align*}
   \norm{f_c}_\infty\lesssim \min\left\{\frac{V_0+1}{t^\frac{1}{4}},\frac{V_0^\frac{3}{2}+1}{t^\frac{3}{8}}\right\}.
\end{align*}
Hence we observe relaxation
at the long-term rate $t^{-\frac{3}{8}}$ in $L^\infty$ to the $L^2$-closest kink and no rate of convergence to the centered kink.
\end{remark}

\subsection{Remarks on the method}

Our method uses a combination of basic algebraic relationships among the quantities $\energygap, c, \dissipation, \M,$ and $\V$ and a duality argument to establish boundedness of $\M$ or $\V$ in terms of its initial value. For future reference, we point out that, according to \eqref{diff}, the energy gap is decreasing. We will use this often together with the assumption that the initial energy gap is order one:
\begin{align}\label{eq:Edecreassing}
	\E\leq \E_{0}\lesssim 1\qquad\text{for all }t\geq 0.
\end{align}
Next we recall the following estimates from \cite[Lemma 1.3]{OW}.
\begin{lemma}\label{l:eed}
	Suppose that $u$ is smooth and there exists $\epsilon>0$ such that $\E(u)\leq 2\co-\epsilon$. Then
	\begin{align}
		\energygap&\sim \int f_{c}^{2}+f_{cx}^{2}\,\dd x,\label{eq:scale_E}\\
		\D&\sim \int f_{cx}^{2}+f_{cxx}^{2}+f_{cxxx}^{2}\,\dd x.\label{eq:est_D}
	\end{align}
\end{lemma}
We observe that \eqref{eq:scale_E} and \eqref{eq:est_D} together with \eqref{eq:Edecreassing} imply
\begin{align}\label{linf}
	\norm{f_c}_\infty\lesssim \min\left\{\E^{\frac{1}{2}},(\E\D)^{\frac{1}{4}}\right\}\lesssim 1\qquad\text{and}\qquad
	\norm{f_{cxx}}_\infty\lesssim \dissipation^{\frac{1}{2}},
\end{align}
where here and below $\norm{\cdot}_\infty$ refers to the supremum over $\R$ or $(0,\infty)$, depending on the context.

Other important algebraic relationships that we will establish and use are the Nash-type inequalities
\begin{align}
		\E & \lesssim \D^{\frac{3}{5}}(M+1)^{\frac{4}{5}}\qquad\text{and}\qquad
	\energygap \lesssim \dissipation^{\frac{1}{3}}(V+1)^{\frac{4}{3}}.\label{nashy}
	\end{align}
(See \eqref{al.1} and \eqref{eq:nash} below.) Such inequalities are similar to the kind of nonlinear estimates used as a first step by Nash in his classical derivation of heat kernel bounds.
We will see that it is straightforward to deduce from \eqref{nashy} the decay of the energy gap $\energygap$ in time \emph{in terms of the supremum of $M$ or $V$}, respectively. Hence it will be critical for our argument to show that $M$ and $V$ in fact \emph{remain bounded in terms of their initial values}. This fact we establish via a duality argument similar for instance to that used by Niethammer and Vel\'azquez to show existence of self-similar solutions of Smoluchowski's equation, cf.~\cite[Subsection 1.3]{NV}.

In order to carry out the duality argument, we will need semigroup estimates on the linearized operator $\partial_t-G''(1)\partial_x^2+\partial_x^4$ (or rather its dual). For $M$, it is sufficient to consider the half-space $x>0$. For $V$, however, we have to work on a variable domain $x>\gamma(t)$. In order to treat this case perturbatively, we need Schauder estimates on $x>\gamma(t)$. The corresponding H\"older norm is defined with regard
to the Carnot-Carath\'eodory metric associated to the dual of $\partial_t-G''(1)\partial_x^2+\partial_x^4$:
\begin{align}\label{eq:hoelder_semi_norm}
	 [f]_{\alpha}:=\sup_{t,x}\sup_{s,z}\frac{\abs{f(t+s,x+z)-f(t,z)}}{(\abs{s}+( z^2 \wedge z^4 ))^{\alpha}},
\end{align}
where $\alpha\in (0,\frac{1}{4})$ and $\wedge$ is defined in Notation \ref{N:vee}, below.
Notice that if $f$ is a function of time alone, then  $[f]_{\alpha}$ reduces to the usual H\"older seminorm.

Finally, we note for reference below that the function $f$ defined in \eqref{fc} satisfies
\begin{align}
	f_{t}&=\bigl(G'(v_{c}+f_{c})-G'(v_{c})-f_{cxx}\bigr)_{xx}.\label{eq:ft}
\end{align}

\subsection{Previous results and optimality}\label{ss:previous}

A more elaborate review of previous results and discussion of optimality is available in \cite[Subsections 1.3 and 1.4]{OW}; here we give a brief summary.

Using a renormalization group approach, Bricmont, Kupiainen, and Taskinen \cite{BKT} consider \eqref{ch} with initial data given by
\begin{align*}
	u_{0}(x)=v(x)+f(x),\qquad\text{where}\qquad
	\sup_x\bigl((1+\abs{x}^{p})f(x)\bigr)\leq \delta
\end{align*}
for \emph{some $\delta>0$ sufficiently small} and some $p>3$. They deduce the convergence rate
\begin{align*}
	\sup_{x}\,\abs{u(t,x)-v(x)}\lesssim \frac{1}{t^{\frac{1}{2}}}.
\end{align*}
Carlen, Carvalho, and Orlandi \cite{CCO}, on the other hand, assume bounded second moments and $L^{2}$-closeness to the shifted kink in the form
\begin{align*}
	\int x^{2}(u_{0}-v_{c})^{2}\,\dd x\lesssim 1\qquad\text{and}\qquad
	\int (u_{0}-v_{c})^{2}\,\dd x\leq \delta
\end{align*}
for \emph{some $\delta>0$ sufficiently small}. They show that for any $\epsilon>0$, the energy gap and $L^{1}$ distance to the centered kink satisfy
\begin{align*}
	\E(u)\lesssim \frac{1}{(t+1)^{\frac{9}{13}-\epsilon}}\qquad\text{and}\qquad
	\int \abs{u(t,x)-v(x)}\,\dd x\lesssim \frac{1}{(t+1)^{\frac{5}{52}-\epsilon}}.
\end{align*}
Finally, Howard \cite{H} considers \eqref{ch} for initial data satisfying the spatial decay assumption
\begin{align}
	\abs{u_{0}(x)-v(x)}\leq \delta \frac{1}{1+x^{2}},\label{How1}
\end{align}
for \emph{some $\delta>0$ sufficiently small}. Using a careful spectral analysis of the linearized operator, he establishes
the \textbf{optimal} relaxation rates
\begin{align*}
	\norm{f_{c}}_\infty\lesssim \frac{1}{t+1}\qquad\text{and}\qquad
	\abs{c}\lesssim \frac{1}{t^\frac{1}{2}+1}.
\end{align*}
All three results \cite{BKT, CCO, H} are perturbative in the sense that they assume initial data satisfying a smallness assumption in terms of some sufficiently small parameter $\delta$.
As mentioned above (cf.\ the discussion around \eqref{hed}),  the previous work \cite{OW} deduces decay rates based on the initial energy bound $\Ez\leq 2e_*-\epsilon$ for $\epsilon>0$ and finiteness---\textbf{not smallness}---of the initial $\dot{H}^{-1}$ distance.

In this paper we measure the distance to the slow manifold in a different metric, using either
the first moment $\M$ or the excess mass $\V$ as a measure of distance. As in \cite{OW} for the $\dot{H}^{-1}$ distance, we require only that $M_0$ or $V_0$ is \emph{finite}. No smallness is assumed. On the level of scaling, our functional $M$ represents the same control as Howard's assumption \eqref{How1}, but in the (weaker) $L^1$ sense. The decay result we obtain in this case (cf. \eqref{eq:energydecayM} and \eqref{infdecay}) is the best possible; even compactly supported initial perturbations do not yield faster rates of decay. Again we refer to \cite[Subsection 1.3]{OW} for a formal justification of this claim.

\subsection{Conventions and organization}
We make the following (standard) assumption on the double-well potential.
\begin{assumption}\label{ass:G}
	The double-well potential $G$
	is assumed to satisfy
	\begin{itemize}
		\item
			$G$ is $C^2(\R)$ and even,
		\item
			$G(u)>0$ for $u\not=\pm 1$ and $G(\pm 1)=0$,
		\item
			$G'(u)\leq 0$ for $u\in [0,1]$ and $G''(\pm 1)>0$.
	\end{itemize}
	The assumption that $G$ is even is for convenience and could be relaxed.
\end{assumption}

Often we will make use of the following shorthand notations.
\begin{notation}\label{not:lesssim}
	Throughout the paper we use the notation
	\begin{align*}
		A\lesssim B
	\end{align*}
	if there exists a  constant $C\in (0,\infty)$ such that $A \leq CB$. Unless otherwise noted the constant $C$ depends at most on
	the potential $G$ and the constant $\epsilon>0$ in \eqref{eq:epsilon}.
If $A\lesssim B$ and $B\lesssim A$, we write $A\sim B$.
\end{notation}
\begin{notation}\label{N:vee}
  We use $A\wedge B$ and $A\vee B$ to denote the minimum of $\{A, B\}$ and the maximum of $\{A, B\}$, respectively.
\end{notation}

In Sections \ref{section:M} and \ref{section:V} we give the proofs of \eqref{slowdown1}--\eqref{eq:energydecayM} and \eqref{slowdown2}--\eqref{eq:energydecayV}, respectively. The proofs of the parabolic estimates are given in Section \ref{S:paraproof}. The estimates for $\V$ rely on Schauder estimates on the half-line (cf.\ Proposition \ref{prop:Schauder}), whose proof is deferred to Appendix \ref{section:proof_schauder}.
	
\section{Relaxation under bounded initial first moment}\label{section:M}

Here we collect the ingredients that we will use to prove the relaxation estimate under the assumption of bounded initial first moment, i.e., case $(i)$ of Theorem \ref{t:mainresult}. With the lemmas in hand, we easily deduce \eqref{slowdown1} and \eqref{eq:energydecayM}. The proofs of the supporting lemmas are given in Subsection~\ref{ss:pflemmaM}.

Throughout Section \ref{section:M}, we work \emph{under the assumption} \eqref{eq:epsilon}.

\subsection{Parabolic estimates for $M$}
For the duality argument for $M$ we will need the following (standard) estimates for the solution of a fourth order equation on the half-line.
\begin{proposition}\label{paraboliczeta}
	Let $\zeta$ solve the backwards problem
	\begin{align}\label{eq:zetaequa}
		\left\{
		\begin{aligned}
			\zeta_t+G''(1)\zeta_{xx}-\zeta_{xxxx}&=0		&&\text{on }	&&t\in [0,T),	&&x\in (-\infty,0),\\
			\zeta=\zeta_{xx}&=0								&&\text{for }	&&t\in [0,T),	&& x=0,\\
			\zeta&=\g										&&\text{for }	&&t=T,			&& x\in (-\infty,0], \\
		\end{aligned}
		\right.
	\end{align}
	where $\g$ satisfies $\norm{\g_x}_\infty\leq 1$.
	There holds
	\begin{align}
		\norm*{\zeta_x}_{\infty}&\lesssim 1,\label{zetax}\\
		 \norm*{\zeta_{xx}}_{\infty}&\lesssim\frac{1}{(T-t)^{\frac{1}{2}}}\wedge\frac{1}{(T-t)^{\frac{1}{4}}},\label{zetaxx}\\
		 \norm*{\zeta_{xxx}}_{\infty}&\lesssim\frac{1}{T-t}\wedge\frac{1}{(T-t)^{\frac{1}{2}}},\label{zetaxxx}\\
		 \norm*{\zeta_{xxxx}}_{\infty}&\lesssim\frac{1}{(T-t)^{\frac{3}{2}}}\wedge\frac{1}{(T-t)^{\frac{3}{4}}},\label{zetaxxxx}\\
		 \norm{\zeta_{xxxxx}}_{\infty}&\lesssim\frac{1}{(T-t)^{2}}\wedge\frac{1}{T-t},\label{zetaxxxxx}\\
		\left|\int_{-\infty}^0 \zeta_t (v_c-v)\,\dd x \right|&\lesssim
		\begin{cases}
			\frac{1}{T-t}(c^2+\abs{c})\\
			\left(\frac{1}{(T-t)^{\frac{1}{2}}}+ \frac{1}{(T-t)^{\frac{3}{4}}}\right)\abs{c}.\label{A}
		\end{cases}
	\end{align}
\end{proposition}
For completeness we include a proof in Subsection~\ref{subs:parabolicM}.

\subsection{Nonlinear ingredients and proof of \eqref{eq:energydecayM}}
The heart of the idea is to control the energy in terms of the first moment $M$. For this we will make use of Lemma \ref{l:eed} together with the following Nash-type inequality.

\begin{lemma}[Nash's inequality]\label{l:realnash}
	There holds
	\begin{align}
		\E & \lesssim \D^{\frac{3}{5}}(M+1)^{\frac{4}{5}}.\label{al.1}
	\end{align}
\end{lemma}

We will work on the finite time horizon $t\leq T$ for any fixed $T\in (0,\infty)$. It is convenient to abbreviate
\begin{align}\label{eq:defMb}
	\Mb:=\sup_{t\leq T}\M\vee 1.
\end{align}
Using Lemma \ref{l:realnash}, it is straightforward to deduce a bound on the energy gap in terms of $\Mb$.

\begin{corollary}\label{l:nash}
	The energy gap is bounded above by
	\begin{align}
		\E\lesssim \frac{\Mb^2}{t^{\frac{3}{2}}}\qquad \text{for }t\in[0,T].\label{nash1}
	\end{align}
\end{corollary}
Hence it remains only to control $\Mb$.
Before turning to that task, we introduce two auxiliary results: Lemmas \ref{l:Malg} and \ref{l:dissint}.

\begin{lemma}[Additional algebraic relationships]\label{l:Malg}
	There holds
	\begin{align}
		\abs{c}&\lesssim \E^{\frac{1}{3}}\M^{\frac{1}{3}},\label{al.2}\\
		\abs{f_c(0)}&\lesssim (M+1)^{\frac{1}{6}}\D^{\frac{1}{2}}.\label{pw}
	\end{align}
\end{lemma}

We use Lemma \ref{l:Malg} together with \eqref{eq:Edecreassing} and Corollary \ref{l:nash} to control the shift.

\begin{corollary}\label{c:estimatesM}
	There holds
	\begin{align}
		\abs{c}\lesssim \min\left\{\Mb^{\frac{1}{3}},\frac{\Mb}{t^{\frac{1}{2}}}\right\}. \label{folg}
	\end{align}
\end{corollary}

Our second auxiliary ingredient is the following.
\begin{lemma}[Integral dissipation bound]\label{l:dissint}
	For any $\frac{2}{5}<\exponent\leq 1$, there holds
	\begin{align}
		\int_0^T \D^\exponent\,\dd t\lesssim \Mb^{\frac{4(1-\exponent)}{3}}.\label{dissint}
	\end{align}
\end{lemma}

We are now ready to control $\Mb$.
\begin{proposition}[Duality argument]\label{l:dualM}
	For any $T\in(0,\infty)$, the first moment remains in an order one neighborhood of its initial value
	\begin{align*}
		\Mb\lesssim \M_0+1.
	\end{align*}
\end{proposition}

These ingredients suffice to prove energy decay in the case of bounded first moment.

\begin{proof}[Proof of Theorem \ref{t:mainresult}, case $(i)$]
	Proposition \ref{l:dualM} establishes \eqref{slowdown1}. The estimates in \eqref{eq:energydecayM} follow from a combination of \eqref{eq:Edecreassing}, Corollary \ref{l:nash}, Corollary \ref{c:estimatesM}, and Proposition \ref{l:dualM}.
\end{proof}

\subsection{Proofs of auxiliary results}\label{ss:pflemmaM}

\begin{proof}[Proof of Lemma \ref{l:realnash}]
We appeal to \eqref{eq:scale_E} in the form
\begin{align*}
  \E\lesssim\left(\int_{-\infty}^c+\int_c^\infty\right)(F_{cx}^2+F_{cxx}^2)\, \dd x,
\end{align*}
cf.\ \eqref{fc}. For the first contribution, we invoke the elementary interpolation estimate
	\begin{align}
		\int_{-\infty}^c F_{cx}^2\,\dd x&\lesssim \left(\int_{-\infty}^cF_{cxx}^2\,\dd x\right)^{\frac{3}{5}}\left(\int_{-\infty}^c\abs{F_c}\,\dd x\right)^{\frac{4}{5}},\label{easy1}
	\end{align}
and the corresponding estimate on $(c,\infty)$, which together with \eqref{eq:est_D} and \eqref{eq:definitionM} yield
\begin{align*}
    	\left(\int_{-\infty}^c+\int_c^\infty\right) F_{cx}^{2}\,\dd x
		\lesssim D^{\frac{3}{5}}\M^{\frac{4}{5}}.
	\end{align*}
For completeness we remark that \eqref{easy1} follows from
the classical interpolation estimate
	\begin{align}
		\int_{-\infty}^c F_{cx}^2\,\dd x&\lesssim \left(\int_{-\infty}^cF_{cxx}^2\,\dd x\right)^{\frac{1}{2}} \left(\int_{-\infty}^cF_c^2\,\dd x\right)^{\frac{1}{2}}\label{int}
	\end{align}
and the estimate
	\begin{align}
		\sup_{(-\infty,c)}\abs{F_c}&\lesssim \left( \int_{-\infty}^cF_{cx}^2\,\dd x\right)^{\frac{1}{3}}\left(\int_{-\infty}^c\abs{F_c}\,\dd x\right)^{\frac{1}{3}}.\label{easy2}
	\end{align}
Estimate \eqref{easy2} in turn follows from the elementary facts
	\begin{align*}
		\sup_{(-\infty,c)}\abs{F_{c}}^{2}\lesssim \left(\int_{-\infty}^{c}F_{cx}^{2}\,\dd x\int_{-\infty}^{c}F_{c}^{2}\,\dd x\right)^{\frac{1}{2}}\quad\text{and}\quad
		\int_{-\infty}^{c}F_{c}^{2}\,\dd x\leq \sup_{(-\infty,c)}\abs{F_{c}}\int_{-\infty}^{c}\abs{F_{c}}\,\dd x.
\end{align*}

For the second contribution to $\E$, we employ the simplistic estimate
	\begin{align*}
    \left(\int_{-\infty}^c+\int_c^\infty\right)	 F_{cxx}^{2}\,\dd x\overset{\eqref{eq:scale_E}, \eqref{eq:est_D} }\lesssim D^{\frac{3}{5}}\E^{\frac{2}{5}}
		\overset{\eqref{eq:Edecreassing}}\lesssim D^{\frac{3}{5}}
		\leq D^{\frac{3}{5}}(M+1)^{\frac{4}{5}}.
	\end{align*}

\end{proof}

\begin{proof}[Proof of Corollary \ref{l:nash}]
	From \eqref{diff}, \eqref{eq:defMb}, \eqref{al.1}, and an integration in time, we obtain \eqref{nash1}.
\end{proof}

\begin{proof}[Proof of Lemma \ref{l:Malg}]
Notice that \eqref{diff} and the Cauchy Schwarz inequality implies that $u-u_0$ is in $\dot{H}^{-1}$ for any finite time and hence, that  $\int u-u_0\,\dd x=0$. Because of the condition $\int u_0-v\,\dd x=0$, there holds  $\int u-v\,\dd x=0$, which we use to calculate
	\begin{align}\notag
		2c=\int v-v_c\,\dd x=\int u-v+v-v_c\,\dd x\overset{\eqref{fc}}=\int f_c\,\dd x\overset{\eqref{fc}}=F_c\left(c_-\right)-F_c\left(c_+\right),
	\end{align}
so that $\abs{c}\leq  \sup \abs{F_c}$. This fact together with \eqref{easy2}, \eqref{eq:scale_E}, and \eqref{eq:definitionM} implies \eqref{al.2}.

	For \eqref{pw} we use \eqref{eq:ELvc} and argue as in \cite[proof of Lemma 2.1]{OW} to obtain
\begin{align}
\abs{f_{c}(c)}\lesssim \D^{\frac{1}{2}}.\label{fcD}
\end{align}
We then use this bound to estimate
	\begin{eqnarray}
		\abs{f_{c}(0)}&\leq& \abs{f_{c}(0)-f_{c}(c)}+\abs{f_{c}(c)}\leq \left(\abs{c}\int f_{cx}^2\,\dd x\right)^\frac{1}{2}+\abs{f_{c}(c)}\notag\\
		&\overset{\eqref{eq:est_D}}\lesssim &\bigl((\abs{c}+1)D\bigr)^{\frac{1}{2}}
		\overset{\eqref{al.2},\eqref{eq:Edecreassing} }\lesssim (M+1)^{\frac{1}{6}}D^{\frac{1}{2}}.\label{fcsmall}
	\end{eqnarray}
\end{proof}

\begin{proof}[Proof of Lemma \ref{l:dissint}]
	Pick a  $\exponentt\in \left(\frac{1}{\exponent}-1,\frac{3}{2}\right)$ and let $\tau\in(0,T)$, to be optimized later.
	On the one hand, we  note that using
	\begin{align*}
		t^{\exponentt}\D\overset{\eqref{diff}}=\frac{\dd}{\dd t}\left(-t^{\exponentt}\E\right)+\exponentt t^{\exponentt-1}\E
	\end{align*}
	there holds
	\begin{align*}
		\MoveEqLeft
		\int_{\tau}^{T}\D^{\exponent}\,\dd t\leq \left(\int_{\tau}^{T}\frac{1}{t^{\frac{\exponent\exponentt}{1-\exponent}}}\,\dd t\right)^{1-\exponent}\left(\int_{\tau}^{T}t^{\exponentt}\D\,\dd t\right)^{\exponent}
		\lesssim \frac{1}{\tau^{\exponent\exponentt-(1-\exponent)}}\left(\tau^{\exponentt}\E(\tau)+\int_{\tau}^{T}t^{\exponentt-1}\E\, \dd t\right)^{\exponent}\\
		&\overset{\eqref{nash1}}\lesssim \frac{1}{\tau^{\exponent\exponentt-(1-\exponent)}}\left(\tau^{\exponentt}\frac{\Mb^{2}}{\tau^{\frac{3}{2}}}+\Mb^{2}\int_{\tau}^{T}t^{\exponentt-\frac{5}{2}}\, \dd t\right)^{\exponent}
		\lesssim \frac{1}{\tau^{\exponent\exponentt-(1-\exponent)}}\left(\frac{\Mb^{2}}{\tau^{\frac{3}{2}-\exponentt}}\right)^{\exponent}
		=\frac{\Mb^{2\exponent}}{\tau^{\frac{5\exponent}{2}-1}}.
	\end{align*}
	On the other hand, we have the estimate
	\begin{align}
		\int_{0}^{\tau}\D^{\exponent}\,\dd t\leq \tau^{1-\exponent}\left(\int_{0}^{\tau}\D\,\dd t\right)^{\exponent}\overset{\eqref{diff}}\lesssim \tau^{1-\exponent}\E_{0}^{\exponent}\overset{\eqref{eq:epsilon}}\lesssim \tau^{1-\exponent}.\label{otherhand}
	\end{align}
	Combining the previous two estimates yields
	\begin{align*}
		\int_{0}^{T}\D^{\exponent}\,\dd t\lesssim  \tau^{1-\exponent}+\frac{\Mb^{2\exponent}}{\tau^{\frac{5\exponent}{2}-1}}.
	\end{align*}
	Optimizing in $\tau$ establishes \eqref{dissint}. (Notice that in the event the optimal $\tau$ exceeds $T$, using \eqref{otherhand} on all of $(0,T)$ suffices.)
\end{proof}

\begin{proof}[Proof of Proposition \ref{l:dualM}]
We begin by introducing
\begin{align*}
  F(x):=
    \int_{-\infty}^x f(y)\,\dd y.
\end{align*}
Using $F$, we define the following simpler stand-in for $\M$ (cf.\ \eqref{eq:definitionM}):
\begin{align*}
   \widetilde{M}:=\int_{-\infty}^\infty\abs*{F}\,\dd x.
\end{align*}
That $\Mt$ is a suitable stand-in for $M$ is contained in the following lemma.
\begin{lemma}[Link between $\M$ and $\Mt$]\label{l:mm}
  There holds
  \begin{align*}
		\Mt\lesssim \M+1\quad\text{and}\quad \M\lesssim \Mt+1.
  \end{align*}
\end{lemma}
\begin{proof}[Proof of Lemma \ref{l:mm}]
We remark that because of $\int f(y)\,\dd y=0$, we have $F(x)=-\int_x^\infty f(y)\,\dd y$
for all $x\in\R$.
Using the identity $f_c-f=v-v_c$ (cf.\ \eqref{fc}), we estimate
	\begin{align*}
		\abs*{F_c(x)-F(x)}\lesssim
		\begin{cases}
\abs{c} &\text{for }\abs{x}\leq 2 \abs{c}\\
			\frac{1}{x^2}&\text{for }\abs{x}\geq 2 \abs{c},
					\end{cases}
	\end{align*}
where there is nothing special about $\frac{1}{x^2}$ except that it is integrable at infinity.
	We deduce
	\begin{align*}
		\abs*{\M-\Mt}&\leq \int_{-\infty}^\infty \abs{F_c-F}\,\dd x\lesssim c^2+1.
	\end{align*}
	Invoking \eqref{eq:Edecreassing} and \eqref{al.2} yields
	\begin{align}
		\abs*{\M-\Mt}\lesssim \M^{\frac{2}{3}}+1.\label{ded.1}
	\end{align}
	On the one hand, Young's inequality yields
	$
		\Mt\lesssim \M+1.
	$
	On the other hand we deduce from \eqref{ded.1} that
	$
		\M\lesssim \Mt+ \M^{\frac{2}{3}}+1,
	$
	so that another application of Young's inequality gives
	$
		\M\lesssim \Mt+1.
	$
\end{proof}

We will use a duality argument for $\widetilde{M}$ to establish uniform control on $M$.
According to Lemma \ref{l:mm}, it suffices to show
\begin{align*}
		\sup_{t\leq T}\Mt\lesssim \Mt(0)+1.
	\end{align*}
	We will make use of the dual formulation
	\begin{align*}
		\Mt=\sup\left\{\int f\psi\,\dd x\colon \abs{\psi_x}\leq 1\text{ and }\psi(0)=0  \right\},
	\end{align*}
	which follows from
	\begin{align*}
		\int f\psi\,\dd x=-\int_{-\infty}^0 F\psi_x\,\dd x-\int_0^{\infty} F\psi_x\,\dd x.
	\end{align*}
	Without loss we focus on $x<0$; the estimates on $(0,\infty) $ follow analogously. We now fix $T>0$ and let $\zeta$ be as in Proposition~\ref{paraboliczeta}.
	We calculate
	\begin{align*}
		\frac{\dd}{\dd t}\int_{-\infty}^0 f\zeta\,\dd x=\int_{-\infty}^0 \zeta_t f+\zeta f_t\,\dd x.
	\end{align*}
	Substituting $f=f_c+v_c-v$ and $f_t=u_t=(G'(u)-G'(v_c)-f_{cxx})_{xx}$ (cf.\ \eqref{eq:ft}) and using  equation~\eqref{eq:zetaequa} for $\zeta$, we can rewrite this as
	\begin{align*}
		\frac{\dd}{\dd t}\int_{-\infty}^0 f\zeta\,\dd x&=\int_{-\infty}^0 \zeta_t(v_c-v)\,\dd x\\
		&\quad+\int_{-\infty}^0 -G''(1)\zeta_{xx}f_c+\zeta\Big(G''(v_c)f_c\Big)_{xx}\,\dd x\\
		&\quad+\int_{-\infty}^0 \zeta\Big(G'(u)-G'(v_c)-G''(v_c)f_c\Big)_{xx}\,\dd x\\
		&\quad+\int_{-\infty}^0 \zeta_{xxxx}f_c-\zeta f_{cxxxx}\,\dd x=:I_1+I_2+I_3+I_4.
	\end{align*}
	We will now collect estimates for the time integral of terms $I_1$ to $I_4$ for $T$ large and small.
	By Young's inequality and Lemma \ref{l:mm}, it suffices to show
\begin{align*}\displaystyle
  \int_0^T\left(I_1+\ldots+I_4\right)\dd t\lesssim
  \begin{cases}
    \frac{\ln T}{T^\frac{1}{4}}\,\Mb +\Mb^\frac{5}{6}& \text{for }T\geq 2,\\
    \left(T^\frac{3}{4}+1\right)\left(\Mb^\frac{5}{6}+1\right)& \text{for all }T.
  \end{cases}
\end{align*}
We think of (and refer to) the case $T\geq 2$ as the ``large time'' case and the bound for general $T$ as the case of ``order one times.''
	
	\uline{Term $I_1$}: We estimate term $I_1$ as follows.
We begin by considering $T\geq 2$. From \eqref{folg} we deduce that
	\begin{align}
		 c^2+\abs{c}=(\abs{c}^{\frac{3}{2}}+\abs{c}^{\frac{1}{2}})\abs{c}^{\frac{1}{2}}\overset{\eqref{folg},\eqref{eq:defMb}}\lesssim \left(\Mb^{\frac{1}{3}}\right)^{\frac{3}{2}}\left(\frac{\Mb}{t^{\frac{1}{2}}}\right)^{\frac{1}{2}}=\frac{\Mb}{t^{\frac{1}{4}}}.\label{c2}
	\end{align}
For $t\in[0,T-1]$, we estimate
	\begin{align*}
	\int_0^{T-1}\abs*{\int_{-\infty}^0\zeta_t(v_c-v)\,\dd x}\,\dd t\overset{\eqref{A},\eqref{c2}}\lesssim
\Mb\int_0^{T-1}\frac{1}{t^{\frac{1}{4}}}\frac{1}{T-t}\,\dd t
	\lesssim\Mb \frac{\ln T}{T^{\frac{1}{4}}}.
	\end{align*}
	For the terminal layer $[T-1,T]$, we find
	\begin{align*}
		\int_{T-1}^{T}\abs*{\int_{-\infty}^0\zeta_t(v_c-v)\,\dd x}\,\dd t\overset{\eqref{A},\eqref{folg}}\lesssim\Mb\int_{T-1}^T\frac{1}{t^{\frac{1}{2}}}\frac{1}{(T-t)^{\frac{3}{4}}}\,\dd t
		\lesssim \frac{\Mb}{T^{\frac{1}{2}}}.
	\end{align*}
	For order one times, we use
	\begin{eqnarray*}
		\int_0^{T}\abs*{\int_{-\infty}^0\zeta_t(v_c-v)\,\dd x}\,\dd t&\overset{\eqref{A}}\lesssim&\int_0^T \left(\frac{1}{(T-t)^{\frac{1}{2}}}+\frac{1}{(T-t)^{\frac{3}{4}}}\right)\abs{c}\,\dd t\\\
		&\overset{\eqref{folg}}\lesssim& \left(T^{\frac{1}{2}}+T^{\frac{1}{4}}\right)\Mb^{\frac{1}{3}}.
	\end{eqnarray*}

	\uline{Term $I_2$}:
	For the ``second order elliptic term'' $I_2$, we integrate by parts (using $\zeta=0$ at $x=0$) to obtain
	\begin{align}
		I_2&=\int_{-\infty}^0\left(-G''(1)\zeta_{xx}f_c-\zeta_x\big(G''(v_c)f_c\big)_x\right)\,\dd x\notag\\
		&=\int_{-\infty}^0\big(G''(v_c)-G''(1)\big)\zeta_{xx}f_c\,\dd x-\zeta_xG''(v_c)f_c\big |_{x=0}.\label{bulkandsurface}
	\end{align}
	For the bulk term we  estimate
	\begin{eqnarray}
		\abs*{\int_{-\infty}^0\big(G''(v_c)-G''(1)\big)\zeta_{xx}f_c\,\dd x}&\leq&\norm{\zeta_{xx}}_{\infty}\int_{-\infty}^0 \abs*{\big(G''(v_c)-G''(1)\big)\,f_c}\,\dd x\notag\\
		&\overset{\eqref{eq:scale_E}, \eqref{linf}}\lesssim& \norm{\zeta_{xx}}_{\infty}\E^{\frac{1}{2}},\label{starone}
	\end{eqnarray}
where we have used the exponential, and thus in particular square integrable, tails of $f_c$.
	Integrating in time, we apply \eqref{zetaxx} and \eqref{nash1} to obtain
	\begin{align}\label{eq:T3/4t1/2}
		\int_0^T \abs*{\int_{-\infty}^0\big(G''(v_c)-G''(1)\big)\zeta_{xx}f_c\,\dd x}\,\dd t\lesssim \Mb\int_0^T\frac{1}{(T-t)^{\frac{1}{2}}}\frac{1}{t^{\frac{3}{4}}}\,\dd t\lesssim\Mb\,\frac{1}{T^{\frac{1}{4}}},
	\end{align}
which we use for $T\geq 2$.
	For order one times, we use square integrability of $f_c$, \eqref{linf}, and the other right-hand side term in \eqref{zetaxx} to deduce
	\begin{align*}
		\int_0^T \abs*{\int_{-\infty}^0\big(G''(v_c)-G''(1)\big)\zeta_{xx}f_c\,\dd x}\,\dd t\lesssim\int_0^T\frac{1}{(T-t)^{\frac{1}{4}}}\,\dd t\lesssim T^{\frac{3}{4}}.
	\end{align*}
	We turn now to the boundary term in \eqref{bulkandsurface}. Here we argue
	\begin{eqnarray*}
		\abs*{ \zeta_xG''(v_c)f_c\big |_{x=0}}&\lesssim&\norm{\zeta_{x}}_{\infty}\abs{f_c(0)}
		\overset{\eqref{zetax},\eqref{pw}}\lesssim \Mb^{\frac{1}{6}}D^{\frac{1}{2}}.
	\end{eqnarray*}
	We use \eqref{dissint} to deduce
	\begin{align*}
		\int_0^T\abs*{\zeta_xG''(v_c)f_c\big |_{x=0}}\,\dd t\lesssim\Mb^{\frac{5}{6}}.
	\end{align*}

	\uline{Term $I_3$}:
	We now address the ``nonlinear term'' $I_3$, which we reformulate as
	\begin{align*}
		I_3&=\int_{-\infty}^0 \zeta_{xx}\Big(G'(u)-G'(v_c)-G''(v_c)f_c\Big)\,\dd x-\zeta_x\Big(G'(u)-G'(v_c)-G''(v_c)f_c\Big)\big |_{x=0}.
	\end{align*}
	For the bulk term we estimate
	\begin{align}
		\abs*{ \int_{-\infty}^0 \zeta_{xx}\Big(G'(u)-G'(v_c)-G''(v_c)f_c\Big)\,\dd x}\overset{\eqref{linf}}\lesssim\norm{\zeta_{xx}}_{\infty}\int_{-\infty}^0 f_c^2\,\dd x
		\overset{\eqref{eq:scale_E}}\lesssim\norm{\zeta_{xx}}_{\infty}\,\E.\label{startwo}
	\end{align}
	By using $\E\lesssim 1$ to write $\E\lesssim\E^{1/2}$, we can treat this term as we did the bulk term from $I_2$.

	For the boundary term, we use
	\begin{align*}
		\abs*{\zeta_x\Big(G'(u)-G'(v_c)-G''(v_c)f_c\Big)\big |_{x=0}}\overset{\eqref{linf}}\lesssim\norm{\zeta_{x}}_{\infty}\,f_c^2(0)\overset{\eqref{linf}}\lesssim\norm{\zeta_{x}}_{\infty}\abs{f_c(0)},
	\end{align*}
	which we handle as we did the boundary term of $I_2$.

	\uline{Term $I_4$}:
	Finally, we turn to the ``fourth order elliptic term'' $I_4$, which we integrate by parts to express as
	\begin{align}
I_4		=\zeta_{xxx}\,f_{c}\big|_{x=0}+\zeta_x\,f_{cxx}\big|_{x=0}.\label{bdy12}
	\end{align}
	For the second boundary term, we estimate
	\begin{align*}
		 \abs*{\zeta_x\,f_{cxx}\big|_{x=0}}\overset{\eqref{zetax}}\lesssim\norm{f_{cxx}}_{\infty}\overset{\eqref{linf}}\lesssim \D^{\frac{1}{2}}
	\end{align*}
	and proceed as for the boundary term in $I_2$. For the first boundary term in \eqref{bdy12}, on the other hand, we observe
	\begin{align*}
		 \abs*{\zeta_{xxx}\,f_{c}\big|_{x=0}}\overset{\eqref{linf}}\lesssim\norm{\zeta_{xxx}}_{\infty}\E^{\frac{1}{2}}.
	\end{align*}
	Integrating in time, we estimate
	\begin{align*}
		\int_0^{T}\norm{\zeta_{xxx}}_{\infty}\E^{\frac{1}{2}}\,\dd t
	  	\overset{\eqref{zetaxxx},\eqref{nash1}}\lesssim \Mb\int_0^{T}\frac{1}{(T-t)^{\frac{1}{2}}}\frac{1}{t^{\frac{3}{4}}}\,\dd t
		\lesssim \Mb \frac{1}{T^{\frac{1}{4}}}
	\end{align*}
	for $T\geq 2$, while for order one times, we use
	\begin{align*}
		\int_0^{T}\norm{\zeta_{xxx}}_{\infty}\E^{\frac{1}{2}}\,\dd t
		\overset{\eqref{diff},\eqref{zetaxxx}}\lesssim \int_0^{T}\frac{1}{(T-t)^{\frac{1}{2}}}\,\dd t
		\lesssim T^{\frac{1}{2}}.
	\end{align*}
\end{proof}

\section{Relaxation under bounded initial excess mass}\label{section:V}
As in the previous section, we begin by collecting the ingredients that we will use to prove the relaxation estimate under the assumption of bounded initial excess mass, cf.\  \eqref{eq:energydecayV} of Theorem~\ref{t:mainresult}.
With these ingredients in hand, the proof of \eqref{eq:energydecayV} follows easily. The proofs of the supporting lemmas are given in
Subsection~\ref{ss:pflemmaV}.

As in Section \ref{section:M}, we work throughout Section \ref{section:V} \emph{under the assumption} \eqref{eq:epsilon}.

\subsection{Parabolic estimates for $V$}\label{se:parabolic_estimates}

In the duality argument for $V$, we will need that solutions of the fourth order equation on a time-dependent domain satisfy the same estimates as on the half-line. It is in order to show this that one needs Schauder theory, to treat the motion perturbatively. Motion of the boundary at rate $t^{\frac{1}{2}}$ (on large temporal scales) would be critical; in our application, it is sufficient to take $t^\frac{1}{4}$, which is subcritical.

\begin{proposition}\label{le:regularity}
	For any $C_1\in (0,\infty)$, there exists $\Lambda\in[1,\infty)$ such that the following holds. Consider the curve
	\begin{align*}
		\gamma(t):=c(T)-C_1 (T-t+\Lambda)^{\frac{1}{4}},
	\end{align*}
	and let $\zeta$ satisfy	the backwards problem
	\begin{align}\label{eq:zetagamma}
		\left\{
		\begin{aligned}
			\zeta_t+G''(1)\zeta_{xx}-\zeta_{xxxx}&=0		&&\text{on }	&&t\in [0,T),	&&x\in (-\infty,\gamma(t)),\\
			\zeta=\zeta_{xx}&=0								&&\text{for }	&&t\in [0,T),	&& x=\gamma(t),\\
			\zeta&=\g										&&\text{for }	&&t=T,			&& x\in (-\infty,0], \\
		\end{aligned}
		\right.
	\end{align}
	where $\g$ satisfies $\norm{\g}_\infty\leq 1$.
	There holds
	\begin{align}
		\norm*{\zeta}_{\infty}&\lesssim 1,\label{zetag}\\
		\norm*{\zeta_{x}}_{\infty}&\lesssim \frac{1}{(T-t)^{\frac{1}{2}}}\wedge\frac{1}{(T-t)^{\frac{1}{4}}},\label{zetagx}\\
		\norm*{\zeta_{xx}}_{\infty}&\lesssim \frac{1}{T-t}\wedge\frac{1}{(T-t)^{\frac{1}{2}}},\label{zetagxx}\\
		 \norm*{\zeta_{xxx}}_{\infty}&\lesssim\frac{1}{(T-t)^{\frac{3}{2}}}\wedge\frac{1}{(T-t)^{\frac{3}{4}}},\label{zetagxxx}\\
		 \norm*{\zeta_{xxxx}}_{\infty}&\lesssim\frac{1}{(T-t)^{2}}\wedge\frac{1}{T-t}.\label{zetagxxxx}
	\end{align}
\end{proposition}
\begin{remark}
  As is clear from the proof, it suffices to take $\Lambda$ large with respect to $C_1^4$.
\end{remark}
We prove the proposition in Subsection~\ref{subsection:parabolicvariable}.

\subsection{Nonlinear ingredients and proof of \eqref{eq:energydecayV}}
As in Section~\ref{section:M}, an important observation is the following Nash-type inequality.

\begin{lemma}[Nash's inequality]\label{le:nash}
	There holds
	\begin{align}\label{eq:nash}
		\energygap \lesssim \dissipation^{\frac{1}{3}}(V+1)^{\frac{4}{3}}.
	\end{align}
\end{lemma}

In analogy to the case of first moments, we momentarily fix a time horizon $T\in (0,\infty)$ and define
\begin{align*}
	\Vb:=\sup_{t\leq T}\V\vee 1.
\end{align*}
From \eqref{eq:nash} we deduce control on the energy gap $\E$ in terms of $\Vb$.

\begin{lemma}\label{le:decayEV}
	The energy gap is bounded above by
	\begin{align}\label{eq:decay_E}
		\energygap\lesssim \frac{\bar{\V}^{2}}{t^{\frac{1}{2}}}\qquad \text{for }t\in[0,T].
	\end{align}
\end{lemma}

It remains to control $\bar{\V}$.
As in the previous section, we first derive an auxiliary dissipation estimate.
\begin{lemma}[Integral dissipation bound]\label{l:dissint2}
	For any $\frac{2}{3}<\exponent\leq 1$,	there holds
	\begin{align}
		\int_0^T \D^\exponent\,\dd t\lesssim \Vb^{4(1-\exponent)}.
\label{dissint2}
	\end{align}
\end{lemma}

We now turn to the duality argument
for control of  $\Vb$.
\begin{proposition}[Duality argument]\label{le:Vbbounded}
	For any $T\in(0,\infty) $ there holds
	\begin{align}\label{eq:Vbbounded}
		\Vb\lesssim \V_{0}+1.
	\end{align}
\end{proposition}

These ingredients suffice to establish energy decay in the case of bounded excess mass.

\begin{proof}[Proof of Theorem \ref{t:mainresult}, case (ii)]
	Proposition~\ref{le:Vbbounded} establishes \eqref{slowdown2}.
	The first estimate in \eqref{eq:energydecayV} follow from the combination of \eqref{eq:Edecreassing}, Lemma \ref{le:decayEV}, and Proposition \ref{le:Vbbounded}. For the second estimate in \eqref{eq:energydecayV}, we observe
$2c=\int (v-v_c)\,\dd x=\int f_c\,\dd x$, since $v-v_c=f_c-f$ and the integral of $f$ vanishes (cf. proof of Lemma~\ref{l:Malg}). Hence $2|c|\leq V$ and Proposition~\ref{le:Vbbounded} completes the proof.
\end{proof}

\subsection{Proofs of auxiliary results}\label{ss:pflemmaV}

\begin{proof}[Proof of Lemma \ref{le:nash}]
	We want to estimate $\energygap$ via \eqref{eq:scale_E}.
	Using \eqref{eq:est_D} and \eqref{easy2} (applied to $f_c$), we obtain
	\begin{align*}
		\int f_{c}^{2}\,\dd x\leq \norm{f_{c}}_{\infty}\int \abs{f_{c}}\,\dd x\overset{\eqref{easy2}}\lesssim
		\left(\int f_{cx}^2 \,\dd x\right)^{\frac{1}{3}} \left(\int \abs{f_{c}}\,\dd x\right)^{\frac{4}{3}}
		\overset{\eqref{eq:est_D},\eqref{eq:definitionV}}\lesssim D^{\frac{1}{3}}V^{\frac{4}{3}}.
	\end{align*}
	For the gradient term we infer
	\begin{align*}
		\int f_{cx}^{2}\,\dd x\overset{\eqref{eq:scale_E}, \eqref{eq:est_D} }\lesssim D^{\frac{1}{3}}\E^{\frac{2}{3}}
		\overset{\eqref{eq:Edecreassing}}\lesssim D^{\frac{1}{3}}
		\leq D^{\frac{1}{3}}(V+1)^{\frac{4}{3}}.
	\end{align*}
By \eqref{eq:scale_E}, these two estimates yield \eqref{eq:nash}.
\end{proof}

The proof of Lemma \ref{le:decayEV}, which is part of Nash's classical approach to heat kernel bounds,
is analogous to the proof of  Corollary \ref{l:nash} and we omit it.

\begin{proof}[Proof of Lemma \ref{l:dissint2}]
For any $\frac{2}{3}<\exponent\leq 1$, we proceed as in the proof of Lemma \ref{l:dissint} (substituting \eqref{eq:decay_E} instead of \eqref{nash1}) to obtain
\begin{align*}
  \int_0^T D^\exponent \,\dd t\lesssim \tau^{1-\exponent}+\frac{\Vb^{2\exponent}}{\tau^{\frac{3\exponent}{2}-1}}.
\end{align*}
Optimization in $\tau$ gives \eqref{dissint2}.
\end{proof}

\begin{proof}[Proof of Proposition \ref{le:Vbbounded}]

We begin with a rough bound on the motion of zeros.
\begin{lemma}\label{le:roughboundonc}
	For any $0\leq t\leq T$, there holds
	\begin{align}\label{eq:roughboundc}
		\abs{c(T)-c(t)}\lesssim \energygap^{\frac{1}{2}}(t)(T-t+1)^{\frac{1}{4}}\overset{\eqref{eq:Edecreassing}}\lesssim (T-t+1)^{\frac{1}{4}}.
	\end{align}
\end{lemma}
\begin{proof}[Proof of Lemma \ref{le:roughboundonc}]
	Let $\eta:\R\to[0,1]$ be a compactly supported cut-off function such that
	\begin{align}
\abs{\eta_{x}}\lesssim \frac{1}{L}\quad\text{and}\quad \eta\equiv 1 \text{ on }\left[-L+c(t)\wedge c(T), c(t)\vee c(T)+L \right].\notag
	\end{align}
We  write
	\begin{align*}
		2\abs{c(T)-c(t)}&=\abs*{\int v_{c}(T)-v_{c}(t)\,\dd x}\\
		&\leq \abs*{\int \eta \bigl(v_{c}(T)-v_{c}(t)\bigr)\,\dd x}+\abs*{\int (1-\eta) \bigl(v_{c}(T)-v_{c}(t)\bigr)\,\dd x}.
	\end{align*}
Since $L\gg 1$ yields
\begin{align*}
  \abs*{\int (1-\eta)(v_c(T)-v_c(t))\dd x}\ll\abs{c(T)-c(t)},
\end{align*}
we can absorb the second right-hand side term and deduce
	\begin{align*}
		\MoveEqLeft
		\abs{c(T)-c(t)}
		\leq \abs*{\int \eta \bigl(v_{c}(T)-v_{c}(t)\bigr)\,\dd x}.
	\end{align*}
	We now use the Cauchy-Schwarz inequality and duality to estimate
	\begin{eqnarray}
		\abs{c(T)-c(t)}&\leq& \abs*{\int \eta \bigl(v_{c}(T)-v_{c}(t)\bigr)\,\dd x}\notag\\
		&\overset{\eqref{fc}}\leq &\abs*{\int \eta f_{c}(T)\,\dd x}
		+\abs*{\int \eta \bigl(u(T)-u(t)\bigr)\,\dd x}
		+\abs*{\int \eta f_{c}(t)\,\dd x}\notag\\
		&\overset{\eqref{eq:Edecreassing},\eqref{eq:scale_E}}\lesssim &\bigl(L\energygap(t)\bigr)^{\frac{1}{2}}
		+\left(\int \eta_{x}^{2}\,\dd x\; \norm{u(T,\cdot)-u(t,\cdot)}_{\dot{H}^{-1}}^2\right)^{\frac{1}{2}}\notag\\
		&\overset{\eqref{diff}}\lesssim &\bigl(L\energygap(t)\bigr)^{\frac{1}{2}}
		+\left(\frac{T-t}{L} \int_{t}^{T}D \,\dd s\right)^{\frac{1}{2}}\notag\\ 
		&\overset{\eqref{diff}}\lesssim & \energygap^{\frac{1}{2}}(t)\left(L^{\frac{1}{2}}+\left(\frac{T-t}{L}\right)^{\frac{1}{2}}\right).\label{eq:tooptimiseinL}
	\end{eqnarray}
	We obtain \eqref{eq:roughboundc} via optimization of \eqref{eq:tooptimiseinL} in $L$ subject to $L\gg 1$.
\end{proof}

Using this information, we will introduce a simpler stand-in for $\V$. We start by introducing the ($T$-dependent) curves
\begin{align}
	\gamma_-(t):=c(T)-C_1(T-t+\Lambda)^{\frac{1}{4}}\quad\text{and}\quad
	\gamma_+(t):=c(T)+C_1 (T-t+\Lambda)^{\frac{1}{4}}.\label{gammadef}
\end{align}
Here $C_1\in(0,\infty)$ is a fixed constant chosen large enough with respect to the implicit constant in \eqref{eq:roughboundc} so that
\begin{align}
  c(t)-\gamma_-(t)\geq(T-t+\Lambda)^\frac{1}{4}\quad\text{and}\quad \gamma_+(t)-c(t)\geq (T-t+\Lambda)^\frac{1}{4},\label{eq:c-gamma}
\end{align}
for any $\Lambda$ (in the definition of $\gamma_-$) and we fix $\Lambda\in[1,\infty)$ to be the value (belonging to the constant $C_1$ above) from
Proposition \ref{le:regularity}.
We now introduce for given $T$ the functional
\begin{align}
	\Vt:=\int_{-\infty}^{\gamma_-(t)} \abs{u+1}\,\dd x+\int_{\gamma_+(t)}^{\infty} \abs{u-1}\,\dd x.\label{vtdef}
\end{align}
We can think of $\Vt$ as approximating the excess mass in the following sense.
\begin{lemma}[Link between $\V$ and $\Vt$]\label{le:VandVt}
For any $T\in(0,\infty) $, there holds
	\begin{align}\label{eq:VandVt}
		\Vt\lesssim \V+1\quad\text{for all }t\in [0,T]\qquad\text{and}\qquad
		\V(T)\lesssim \Vt(T)+1.
	\end{align}
\end{lemma}
\begin{proof}[Proof of Lemma \ref{le:VandVt}]
	On the one hand, we use the triangle inequality to deduce
	\begin{align*}
		\Vt\leq \int \abs{u-v_{c}}\,\dd x+\int_{-\infty}^{\gamma_-}\abs{v_{c}+1}\,\dd x+\int_{\gamma_+}^{\infty}\abs{v_{c}-1}\,\dd x
		\overset{\eqref{eq:definitionV},\eqref{eq:c-gamma}}\lesssim \V+1.
	\end{align*}
	On the other hand, the second inequality in \eqref{eq:VandVt} follows from
	\begin{align*}
		\MoveEqLeft
		\V(T)\lesssim \Vt(T)+\int_{-\infty}^{\gamma_-(T)} \abs{v_{c}+1}\,\dd x+\int_{\gamma_+(T)}^{\infty} \abs{v_{c}-1}\,\dd x
		+\int_{\gamma_-(T)}^{\gamma_+(T)}\abs{f_{c}}\,\dd x\\
		&\overset{\eqref{eq:c-gamma},\eqref{eq:scale_E}}\lesssim \Vt(T)+1+\E(T)^{\frac{1}{2}}
		\overset{\eqref{eq:Edecreassing}}\lesssim \Vt(T)+1.
	\end{align*}
\end{proof}

	We will establish the estimates on $(-\infty,\gamma_-)$; the estimates on $(\gamma_+,\infty)$ are analogous. Using duality, we write
	\begin{align*}
		\int_{-\infty}^{\gamma_-} \abs{u+1}\,\dd x=\sup\left\{\int \psi(u+1)\,\dd x \colon \psi:(-\infty,\gamma_-)\to\R\text{ such that }\norm{\psi}_\infty\leq 1\right\}.
	\end{align*}
	For simplicity of notation we drop the minus and write $\gamma$ instead of $\gamma_-$ for the rest of the proof. Let $\zeta$ satisfy \eqref{eq:zetagamma}.
	 We compute	
	\begin{eqnarray*}
	\lefteqn{	\frac{\dd}{\dd t}\int_{-\infty}^{\gamma }\zeta (u+1)\,\dd x}\\
&\overset{\zeta (\gamma)=0}=&\int_{-\infty}^{\gamma}\zeta_{  t}(u+1)+\zeta  u_{t}\,\dd x\\
		&\overset{\eqref{fc},\eqref{eq:ft},\eqref{eq:zetagamma}}=&\int_{-\infty}^{\gamma}\zeta_{  t}(v_{c}+1)+\left(-G''(1)\zeta_{xx}+\zeta_{xxxx}\right)f_{c}+\zeta  (G'(u)-G'(v_{c})-f_{cxx})_{xx}\,\dd x\\
		&\overset{\eqref{eq:zetagamma}}=&\int_{-\infty}^{\gamma}\zeta_{  t}(v_{c}+1)\,\dd x
		+\int_{-\infty}^{\gamma}\left(G'(u)-G'(v_{c})-G''(1)f_{c}\right)\zeta_{  xx}\,\dd x\\
		&&\quad -\left(G'(u)-G'(v_{c})-f_{cxx}\right)\zeta_{  x}\Big|_{x=\gamma}
		+f_{c}\zeta_{  xxx}\Big|_{x=\gamma}\\
		&=&I_{1}+I_{2}+I_{3}+I_{4}.
	\end{eqnarray*}
It suffices to show

\begin{align*}\displaystyle
  \int_0^T\left(I_1+\ldots+I_4\right)\dd t\lesssim
  \begin{cases}
    \frac{\ln T}{T^\frac{1}{4}}\,\Vb +\Vb^\delta+1& \text{for }T\geq 2,\\
    T^\frac{3}{4}+1& \text{for all }T
  \end{cases}
\end{align*}
for some $\delta\in(0,1)$.
Then \eqref{eq:VandVt} implies \eqref{eq:Vbbounded}.
Again we refer to the cases $T\geq 2$ and general $T$ as ``large time'' and ``order one times,'' respectively.

	\uline{Term $I_1$}:
		For the first term, we begin by considering $T\geq 2$. We estimate
		\begin{align*}
			\abs*{\int_{-\infty}^{\gamma}\zeta_{  t}(v_{c}+1)\,\dd x}\lesssim \norm{\zeta_t}_\infty \int_{-\infty}^{\gamma}\abs*{v_{c}+1}\,\dd x
			\overset{\eqref{zetagxx},\eqref{zetagxxxx}}\lesssim \frac{\exp\left(-\frac{c-\gamma}{C}\right)}{T-t}
			\overset{\eqref{eq:c-gamma}}\lesssim \frac{\exp\left(-\frac{(T-t)^{\frac{1}{4}}}{C}\right)}{T-t},
		\end{align*}
		where $C$ is a generic, universal constant (depending only on $G$) and we have used exponential tails of $v_{cx}$ order one away from $c$ and the uniform distance between $c$ and $\gamma_-$ afforded by \eqref{eq:c-gamma}.
		Integrating in time, we obtain
		\begin{align*}
			\int_{0}^{T-1}\abs*{\int_{-\infty}^{\gamma}\zeta_{  t}(v_{c}+1)\,\dd x}\,\dd t\lesssim 1.
		\end{align*}
		For the terminal layer $[T-1,T]$, we note
		\begin{eqnarray}
			\abs*{\int_{-\infty}^{\gamma}\zeta_{  t}(v_{c}+1)\,\dd x}
&\overset{\eqref{eq:zetagamma}}=&\abs*{\int_{-\infty}^{\gamma}\left(-G''(1)\zeta_{  xx}+\zeta_{  xxxx}\right)(v_{c}+1)\,\dd x}\label{413}\\
&=&\abs*{\int_{-\infty}^{\gamma}\left(-G''(1)\zeta_{xx}\right)(v_{c}+1)-\zeta_{xxx}v_{cx}\,\dd x +\zeta_{xxx}(v_c+1)\Big|_{x=\gamma}}\notag\\
			&\overset{\eqref{zetagxx},\eqref{zetagxxx}}\lesssim & \frac{1}{(T-t)^{\frac{1}{2}}}+\frac{1}{(T-t)^{\frac{3}{4}}},\label{eq:intzetatest}
		\end{eqnarray}
		from which we deduce
		\begin{align*}
			\int_{T-1}^{T}\abs*{\int_{-\infty}^{\gamma}\zeta_{  t}(v_{c}+1)\,\dd x}\,\dd t\lesssim 1.			
		\end{align*}
		For order one times, we use \eqref{413} to obtain
		\begin{align*}
			\int_{0}^{T}\abs*{\int_{-\infty}^{\gamma}\zeta_{  t}(v_{c}+1)\,\dd x}\,\dd t\lesssim T^{\frac{1}{2}}+T^{\frac{1}{4}}.
		\end{align*}

	\uline{Term $I_2$}:
		We estimate as in \eqref{starone} and \eqref{startwo} above to obtain
		\begin{eqnarray}
			\abs{I_2}	 &=&\abs*{\int_{-\infty}^{\gamma}\Bigl(G'(u)-G'(v_{c})-G''(v_{c})f_{c}+\bigl(G''(v_{c})-G''(1)\bigr)f_{c}\Bigr)\zeta_{  xx}\,\dd x}\notag\\
&\overset{\eqref{eq:scale_E}, \eqref{linf}}\lesssim& \left(\energygap +\energygap^{\frac{1}{2}}\right)\norm{\zeta_{xx}}_{\infty}.\label{kern}
		\end{eqnarray}
		For $T\geq 2$, we use the bound
		\begin{eqnarray*}
		\int_0^T\abs{I_2}\,\dd t&\overset{\eqref{eq:Edecreassing}, \eqref{eq:decay_E},\eqref{zetagxx} }\lesssim& \Vb	 \int_{0}^{T}\frac{1}{t^{\frac{1}{4}}}\left(\frac{1}{T-t}\wedge\frac{1}{(T-t)^{\frac{1}{2}}}\right)\,\dd t\\
			&\lesssim&
\Vb\left(\int_{0}^{T-1}\frac{1}{t^{\frac{1}{4}}}\frac{1}{T-t}\,\dd t
+\int_{T-1}^{T}\frac{1}{T^{\frac{1}{4}}}\frac{1}{(T-t)^{\frac{1}{2}}}\,\dd t\right)
			\lesssim \Vb\frac{\ln T}{T^{\frac{1}{4}}}.
		\end{eqnarray*}
		For order one times, we use the simpler estimate
\begin{align*}
\int_0^T\abs{I_2}\,\dd t\overset{\eqref{kern}}\lesssim  \int_0^T \left(\energygap +\energygap^{\frac{1}{2}}\right)\norm{\zeta_{xx}}_{\infty}\,\dd t\overset{\eqref{eq:Edecreassing},\eqref{zetagxx}}\lesssim T^\frac{1}{2}.
\end{align*}

	\uline{Term $I_3$}:
		For this first boundary term, we estimate
\begin{align}
			\abs{I_3}
			\lesssim \left(\abs{f_{c}}+\abs{f_{cxx}}\right)\Big|_{x=\gamma}\norm{\zeta_{x}}_{\infty}
			\overset{\eqref{zetagx}}\lesssim \left(\abs{f_{c}}+\abs{f_{cxx}}\right)\Big|_{x=\gamma}
\left(\frac{1}{(T-t)^{\frac{1}{2}}}\wedge\frac{1}{(T-t)^{\frac{1}{4}}}\right).\label{gammad}
		\end{align}
For order one times, we use the simplistic estimates
\begin{eqnarray*}
\int_0^T\norm{f_c}_\infty \frac{1}{(T-t)^\frac{1}{4}}\,\dd t & \overset{\eqref{linf}}\lesssim & T^\frac{3}{4}\lesssim 1,\\
\int_0^T\norm{f_{cxx}}_\infty\frac{1}{(T-t)^\frac{1}{4}}\,\dd t&\overset{\eqref{linf}}\lesssim &\int_0^T D^\frac{1}{2}\frac{1}{(T-t)^\frac{1}{4}}\dd t\overset{\eqref{eq:Edecreassing}}\lesssim T^\frac{1}{4}.
\end{eqnarray*}

For  $T\geq 2$, the same estimates suffice for the terminal layer $(T-1,T)$. For the integral over $(0,T-1)$, we argue as in \eqref{fcsmall}, now using \eqref{fcD}, \eqref{eq:roughboundc}, and \eqref{gammadef} to derive
\begin{align*}
  \abs{f_c(\gamma)}\leq \abs{f_c(\gamma)-f_c(c)}+\abs{f_c(c)}
  \overset{\eqref{fcD}}\lesssim\left(\abs{c-\gamma}D\right)^\frac{1}{2}+D^\frac{1}{2}
  \overset{\eqref{eq:roughboundc}, \eqref{gammadef}} \lesssim (T-t+\Lambda)^\frac{1}{8}D^\frac{1}{2}.
\end{align*}
Again recalling \eqref{linf}, we conclude
\begin{align*}
  \abs{f_c(\gamma)}+\abs{f_{cxx}(\gamma)}\lesssim (T-t+\Lambda)^\frac{1}{8}D^\frac{1}{2}.
\end{align*}
We then integrate (cf. \eqref{gammad}):
\begin{align}
\int_0^{T-1} (T-t+\Lambda)^\frac{1}{8}\left(\frac{1}{(T-t)^{\frac{1}{2}}}\wedge\frac{1}{(T-t)^{\frac{1}{4}}}\right)D^\frac{1}{2}\,\dd t
  &\lesssim \int_0^{T-1}\frac{1}{(T-t)^{\frac{3}{8}}}\,D^\frac{1}{2}\,\dd t. \label{object}
\end{align}
Using Lemma \ref{l:dissint2} yields
\begin{align}
\int_0^{T-1}\abs{I_3}\,\dd t\lesssim \Vb^\delta\notag
\end{align}
where $\delta>0$ is any exponent just short of the ``formal''  $\delta=1$ rate. (For instance integrating \eqref{object} using H\"older exponent $\frac{3}{2}$ for the term $D^\frac{1}{2}$ yields $\delta=\frac{2}{3}$.)



	\uline{Term $I_4$}:
		For this second boundary term, we use
		\begin{align*}
			\abs*{f_{c}\zeta_{ xxx}\Big|_{x=\gamma}}
			\overset{\eqref{linf}}\lesssim\norm{\zeta_{xxx}}_{\infty}\E^{\frac{1}{2}}.
		\end{align*}
		We first consider the case $T\geq 2$.
		Integrating in time, we estimate

		\begin{align*}
			\int_{0}^{T}\norm{\zeta_{xxx}}_{\infty}\E^{\frac{1}{2}}\,\dd t
		  	\overset{\eqref{zetagxxx},\eqref{eq:decay_E}}\lesssim \Vb\int_0^{T}\frac{1}{T-t}\frac{1}{t^{\frac{1}{4}}}\,\dd t
			\lesssim \Vb \frac{\ln T}{T^{\frac{1}{4}}}.
		\end{align*}

		In case $T\lesssim 1$, we estimate
		\begin{align*}
			\int_0^{T}\norm{\zeta_{xxx}}_{\infty}\E^{\frac{1}{2}}\,\dd t
			\overset{\eqref{zetagxxx},\eqref{eq:Edecreassing}}\lesssim \int_0^{T}\frac{1}{(T-t)^{\frac{3}{4}}}\,\dd t
			\lesssim T^{\frac{1}{4}}\lesssim 1.
		\end{align*}
\end{proof}

\section{Proofs of parabolic estimates}\label{S:paraproof}
In this section we prove Propositions~\ref{paraboliczeta} and~\ref{le:regularity}. The former is standard but we include it for completeness. The latter requires estimates on a domain with moving boundary; we refer to the comments in Section~\ref{se:parabolic_estimates}.

We will always assume the necessary spatial and temporal decay conditions on the solution and data so that the solution of the partial differential equation is uniquely determined.

\subsection{Proof of the parabolic estimates for  $\M$}\label{subs:parabolicM}

In this subsection we will prove Proposition~\ref{paraboliczeta}.
We set $G''(1)=1$ for notational simplicity.
\begin{proof}[Proof of Proposition~\ref{paraboliczeta}]
	We begin with an odd reflection of $\g$ and then solve the problem on $\R$ via Fourier transform in space.
	The Fourier transform $\hat \zeta$ satisfies
	\begin{align*}
		\hat\zeta_{t}=(k^{2}+k^{4})\hat\zeta\quad\text{and hence}\quad
		\hat\zeta =\hat \g(k)e^{-(k^{2}+k^{4})(T-t)}.
	\end{align*}
	The solution  $\zeta$ can be expressed as
	\begin{align*}
		\zeta(t,x)=\g(x)\ast H^{(2)}(T-t,x)\ast H^{(4)}(T-t,x)
	\end{align*}
	where $H^{(2)}$ is the (second order) heat kernel and $H^{(4)}$ the biharmonic heat kernel, i.e. the fundamental solution of $u_{t}=-u_{xxxx}$.
The heat kernel
	\begin{align}\notag
		H^{(2)}(t,x)= \frac{1}{(4\Pi t)^{\frac{1}{2}}} \exp\left(-\frac{x^{2}}{4\,t}\right),
	\end{align}
	yields the estimates
	\begin{align}\label{eq:Hestimates}
		\norm{H^{(2)}}_{1}\lesssim 1,\,\,\,
		\norm{H^{(2)}_{x}}_{1}\lesssim \frac{1}{t^{\frac{1}{2}}},\,\,\,
		\norm{H^{(2)}_{xx}}_{1}\lesssim \frac{1}{t},\,\,\,
		\norm{H^{(2)}_{xxx}}_{1}\lesssim \frac{1}{t^{\frac{3}{2}}}\,\,\,\text{and}\,\,\,
		\norm{H^{(2)}_{xxxx}}_{1}\lesssim \frac{1}{t^{2}},
	\end{align}
where $\norm{\cdot}_1$ denotes the $L^1$ norm.
	For the biharmonic heat kernel, one can read off from the Fourier transform that
	\begin{align*}
		H^{(4)}(t,x)= \frac{1}{t^{\frac{1}{4}}} H^{(4)}\left(1,\frac{x}{t^{\frac{1}{4}}}\right)
	\end{align*}
	for a Schwartz function $H^{(4)}(1,\cdot)$.
	This implies
	\begin{align}\label{eq:Bestimates}
		\norm{H^{(4)}}_{1}\lesssim 1,\quad
		\norm{H^{(4)}_{x}}_{1}\lesssim \frac{1}{t^{\frac{1}{4}}},\quad
		\norm{H^{(4)}_{xx}}_{1}\lesssim \frac{1}{t^{\frac{1}{2}}},\quad
		\norm{H^{(4)}_{xxx}}_{1}\lesssim \frac{1}{t^{\frac{3}{4}}}\quad\text{and}\quad
		\norm{H^{(4)}_{xxxx}}_{1}\lesssim \frac{1}{t}.
	\end{align}
	Using $\norm{\g_{x}}_{\infty}\leq 1$ and the estimates \eqref{eq:Hestimates} and \eqref{eq:Bestimates} for $H^{(2)}$ and $H^{(4)}$, we estimate
	\begin{align}\label{eq:zetaxsup}
		\norm{\zeta_{x}}_{\infty}\leq  \norm{\g_{x}}_{\infty}\,\norm{H^{(2)}\ast H^{(4)}}_{1} \leq \norm{H^{(2)}}_{1}\,\norm{H^{(4)}}_{1}\lesssim 1.
	\end{align}
	By putting further derivatives on $H^{(2)}$ or $H^{(4)}$, we obtain \eqref{zetaxx} to \eqref{zetaxxxxx}.

	For \eqref{A} we note that from \eqref{eq:zetaequa} and the previous estimates it follows that
	\begin{align*}
		\abs{\zeta_{t}}\lesssim \frac{1}{(T-t)^{\frac{1}{2}}}+\frac{1}{(T-t)^{\frac{3}{4}}}\qquad\text{and}\qquad
		\abs{\zeta_{tx}}\lesssim \frac{1}{T-t}.
	\end{align*}
	Using $\zeta(x=0)=0$, we obtain
	\begin{align*}
		\int_{-\infty}^{0}\zeta_{t}(v_{c}-v)\,\dd x\lesssim
		\begin{cases}
			\frac{1}{T-t}\int_{-\infty}^{\infty}\abs{x}\abs{v_{c}-v}\,\dd x\\
			 \left(\frac{1}{(T-t)^{\frac{1}{2}}}+\frac{1}{(T-t)^{\frac{3}{4}}}\right)\int_{-\infty}^{\infty}\abs{v_{c}-v}\,\dd x.
		\end{cases}
	\end{align*}
On the one hand, $\int\abs{v_c-v}\,\dd x\lesssim \abs{c}$ follows readily from the properties of $v$. On the other hand $\int \abs{x}\abs{v_c-v}\,\dd x\lesssim c^2+\abs{c}$ follows by considering the cases $c\gtrsim 1$ and $c\ll 1$ (and the properties of $v$).
\end{proof}

\subsection{Proof of the parabolic estimates for $V$}\label{subsection:parabolicvariable}
In this subsection we prove Proposition \ref{le:regularity}. It is convenient to introduce the following notation.

\begin{notation}
We will set $G''(1)=1$ for simplicity. Also we will use $\norm{\cdot}$ to denote $\norm{\cdot}_\infty$ and define $$\norm{u,v}:=\max\{\norm{u},\norm{v}\}.$$
In addition we will denote by $$u_{kx}\quad\text{for}\quad k\in\{5,\,6,\,7,\,8\}$$ the $k$-th partial derivative of $u$ with respect to $x$. Finally, we set
\begin{align}
\norm{f}_{t\geq \tau}:=\sup_{x,t\geq\tau} \abs{f(t,x)}.\label{infdef}
\end{align}
\end{notation}

Our approach is to transform \eqref{eq:zetagamma} onto the half-line and treat the resulting advection term perturbatively, i.e., as a right-hand side term, cf.~\eqref{eq:new_eq_zeta}. As usual, one splits the transformed problem into one taking care of the initial data (Lemma~\ref{le:estv}) and one taking care of the right-hand side (Lemma~\ref{le:estu}).

\begin{lemma}\label{le:estv}
	Let $v$ satisfy
	\begin{align}\label{eq:eqv}
		\left\{
		\begin{aligned}
			v_{t}-v_{xx}+v_{xxxx}&=0			&&\text{on }	&&t\in (0,\infty),	&&x\in (0,\infty),\\
			v=v_{xx}&=0								&&\text{for }	&&t\in (0,\infty),	&& x=0,\\
			v&=\g								&&\text{for }	&&t=0,				&& x\in [0,\infty).
		\end{aligned}
		\right.
	\end{align}
	There holds
	\begin{align}\label{eq:estiamteforv2}
		 \sup_{\T>0}\left(\T^{2}\norm{v_{8x},v_{6x},v_{txxxx},v_{xxxx},v_{txx},v_{tt}}_{t\geq \tau}+\T\norm{v_{xxxx},v_{xx},v_{t}}_{t\geq \tau}+\norm{v}_{t\geq \tau}\right)\lesssim \norm{\g}.
	\end{align}
\end{lemma}

For the proof of Lemma~\ref{le:estu}, which is optimal in terms of scaling, the simple semigroup estimates of Proposition~\ref{paraboliczeta} or Lemma~\ref{le:estv} are not sufficient, and we appeal instead to Schauder theory in the form of Proposition~\ref{prop:Schauder}, below.

\begin{lemma}\label{le:estu}
	Let $g(x=0 )=0$ and let $u$ satisfy
	\begin{align}\label{eq:equ}
		\left\{
		\begin{aligned}
			u_{t}-u_{xx}+u_{xxxx}&=g_{x}		&&\text{on }	&&t\in (0,\infty),	&&x\in (0,\infty),\\
			u=u_{xx}&=0								&&\text{for }	&&t\in (0,\infty),	&& x=0,\\
			u&=0								&&\text{for }	&&t=0,				&& x\in [0,\infty).
		\end{aligned}
		\right.
	\end{align}
	There holds
	\begin{align}
		\begin{split}\label{eq:newtoshowu}
			\MoveEqLeft
			\sup_{\tau>0}\left(\T\norm{u_{xxxx},u_{xx},u_{t}}_{t\geq \tau}+\norm{u}_{t\geq \tau}\right)\\
			&\lesssim \sup_{\T>0}\left(\T^{\frac{1}{2}}\wedge\T^{\frac{3}{4}}\right)\left(\T\norm{g_{xxxx},g_{xx},g_{t}}_{t\geq \T}+\norm{g}_{t\geq\tau}\right).
		\end{split}
	\end{align}
\end{lemma}

\begin{corollary}\label{co:estu}
	For $u$ and $g$ as in Lemma \ref{le:estu}, we have that
	\begin{align}
		\begin{split}\label{eq:estu}
			\MoveEqLeft
			\sup_{\T>0}\left(\T^{2}\norm{u_{8x},u_{6x},u_{txxxx},u_{xxxx},u_{txx},u_{tt}}_{t\geq \tau}+\T\norm{u_{xxxx},u_{xx},u_{t}}_{t\geq \tau}+\norm{u}_{t\geq \tau}\right)\\
			&\lesssim \sup_{\T>0}\left(\T^{\frac{1}{2}}\wedge\T^{\frac{3}{4}}\right)\left(\T^{2}\norm{g_{8x},g_{6x},g_{txxxx},g_{xxxx},g_{txx},g_{tt}}_{t\geq \tau}+\T\norm{g_{xxxx},g_{xx},g_{t}}_{t\geq \tau}+\norm{g}_{t\geq \tau}\right).
		\end{split}
	\end{align}
\end{corollary}

We prove Lemmas \ref{le:estv}--\ref{le:estu} and Corollary \ref{co:estu} below after first showing how Proposition \ref{le:regularity} may be deduced from these results.

\begin{proof}[Proof of Proposition \ref{le:regularity}]
	Transforming in time and space via $\hat\zeta(t,x)=\zeta(T-t,\gamma(t)-x)$, $\hat\psi(x)=\psi(\gamma(t)-x)$ (and dropping the hats) leads in   place of \eqref{eq:zetagamma} to the transformed problem
	\begin{align}
		\begin{split}\label{eq:new_eq_zeta}
		\left\{
		\begin{aligned}
			\zeta_{t}-\zeta_{xx}+\zeta_{xxxx}&=g_{x}		&&\text{on }	&&t\in (0,\infty),	&&x\in (0,\infty),\\
			\zeta=\zeta_{xx}&=0								&&\text{for }	&&t\in (0,\infty),	&& x=0,\\
			\zeta&=\g										&&\text{for }	&&t=0,				&& x\in [0,\infty),
		\end{aligned}
		\right.
		\end{split}
	\end{align}
	with
	\begin{align}\label{eqg}
		g:=\dot\gamma\zeta\quad\text{and}\quad
		\gamma(t):=c(T)-C_1 (t+\Lambda)^{\frac{1}{4}}.
	\end{align}
	Then for \eqref{zetag}--\eqref{zetagxxxx}, it suffices to show
	\begin{align}
		\begin{split}\label{eq:schauder}
		\MoveEqLeft
		\sup_{\T>0}\left(\Bigl(\T^{2}\vee\T\Bigr)\norm*{\zeta_{xxxx}}_{t\geq \T}
		+\Bigl(\T^{\frac{3}{2}}\vee\T^{\frac{3}{4}}\Bigr)\norm*{\zeta_{xxx}}_{t\geq \T}
		+\Bigl(\T\vee\T^{\frac{1}{2}}\Bigr)\norm*{\zeta_{xx}}_{t\geq \T}\right.\\
		 &\qquad\qquad\left.+\Bigl(\T^{\frac{1}{2}}\vee\T^{\frac{1}{4}}\Bigr)\norm*{\zeta_{x}}_{t\geq \T}
		+\norm*{\zeta}_{t\geq \T}\right)
		\lesssim 1.
		\end{split}
	\end{align}
	We note that using elementary interpolation (as in the proof of \eqref{eq:newinterpolate2}), it is enough to prove
	\begin{align}\label{eq:intermediate}
		 \sup_{\T>0}\left(\T^{2}\norm{\zeta_{8x},\zeta_{6x},\zeta_{txxxx},\zeta_{xxxx},\zeta_{txx},\zeta_{tt}}_{t\geq \tau}+\T\norm{\zeta_{xxxx},\zeta_{xx},\zeta_{t}}_{t\geq \tau}+\norm{\zeta}_{t\geq \tau}\right)\lesssim 1.
	\end{align}

To this end we write $\zeta=u+v$ with $u$ satisfying \eqref{eq:equ} for $g$ given by \eqref{eqg} and $v$ satisfying \eqref{eq:eqv}.
	According to the triangle inequality, \eqref{eq:estu}, and \eqref{eq:estiamteforv2}, it suffices to establish
	\begin{align*}
		\MoveEqLeft
		 \sup_{\T>0}\left(\T^{\frac{1}{2}}\wedge\T^{\frac{3}{4}}\right)\left(\T^{2}\norm{g_{8x},g_{6x},g_{txxxx},g_{xxxx},g_{txx},g_{tt}}_{t\geq \tau}+\T\norm{g_{xxxx},g_{xx},g_{t}}_{t\geq \tau}+\norm{g}_{t\geq \tau}\right)\\
		&\ll \sup_{\T>0}\left(\T^{2}\norm{\zeta_{8x},\zeta_{6x},\zeta_{txxxx},\zeta_{xxxx},\zeta_{txx},\zeta_{tt}}_{t\geq \tau}+\T\norm{\zeta_{xxxx},\zeta_{xx},\zeta_{t}}_{t\geq \tau}+\norm{\zeta}_{t\geq \tau}\right),
	\end{align*}
where $\ll$ means that the left-hand side can be made smaller than any multiple of the right-hand side by choosing $\Lambda$ sufficiently large.
Using definition~\eqref{eqg}, we observe that it suffices to show
	\begin{align*}
		 \sup_{\T>0}\left(\T^{\frac{1}{2}}\wedge\T^{\frac{3}{4}}\right)\norm*{\T^{2}\dddot\gamma,\T\ddot\gamma,
\dot \gamma}_{t\geq \T}\ll 1,
	\end{align*}
	which  follows for $\Lambda$ sufficiently large with respect to $C_1^4$ from
	\begin{align*}
		\MoveEqLeft
		 \sup_{\T>0}\left(\T^{\frac{1}{2}}\wedge\T^{\frac{3}{4}}\right)\norm*{\T^{2}\dddot\gamma,\T\ddot\gamma,
\dot \gamma}_{t\geq \T}\\
		&\lesssim C_{1}\sup_{\T>0}\T^{\frac{1}{2}}\left(\frac{1}{(\T+\Lambda)^{\frac{3}{4}}}+\frac{\tau}{(\T+\Lambda)^{\frac{7}{4}}}+\frac{
\tau^2}{(\T+\Lambda)^{\frac{11}{4}}}\right)\\
		&\lesssim C_{1}\sup_{\T>0}\frac{1}{(\T+\Lambda)^{\frac{1}{4}}}\leq \frac{C_{1}}{\Lambda^{\frac{1}{4}}}.
	\end{align*}
\end{proof}

It remains to establish the auxiliary results. The proof of Lemma \ref{le:estv} is straightforward.

\begin{proof}[Proof of Lemma \ref{le:estv}]
We deduce \eqref{eq:estiamteforv2} using
	\begin{align}\label{eq:estiamteforv}
		\sup_{\T>0}\left(
		\sum_{k=1}^{8}\Bigl(\T^{\frac{k}{2}}\vee \T^{\frac{k}{4}}\Bigr)\norm*{\partial_x^k v}_{t\geq \T}
		+\T\norm{v_t}_{t\geq\T}+\norm{v}_{t\geq\T}\right)
		\lesssim \norm{\g},
	\end{align}
together with \eqref{eq:eqv} and the triangle inequality.

	To establish \eqref{eq:estiamteforv}, we use odd reflection and proceed as in the proof of Proposition \ref{paraboliczeta}.
Note that we can  bound higher order derivatives by estimating just as we did for the lower order derivatives.
\end{proof}

For the proof of Lemma \ref{le:estu}, we cannot use reflection since $g_{x}$ does not vanish at the origin. Instead, we will use the following half-line Schauder estimates.

\begin{proposition}[Schauder estimates]\label{prop:Schauder}
	Fix $\alpha\in (0,\frac{1}{4})$. Let $u$ satisfy
	\begin{align}\label{eq:t-xx+xxxx}
		\left\{
		\begin{aligned}
			u_{t}-u_{xx}+u_{xxxx}&=f		&&\text{on }	&&t\in \R,	&&x\in (0,\infty),\\
			u=u_{xx}&=0						&&\text{for }	&&t\in \R,	&& x=0.
		\end{aligned}
		\right.
	\end{align}
	Then $u$ satisfies
	\begin{align}\label{eq:Schauder_est}
		[u_{xxxx},u_{xx},u_{t}]_{\alpha}\lesssim [f]_{\alpha},
	\end{align}
	where the constant in \eqref{eq:Schauder_est} depends on $\alpha$ and $[\cdot]_\alpha$ is defined
in~\eqref{eq:hoelder_semi_norm}.
\end{proposition}

\begin{remark}
	In Proposition \ref{prop:Schauder} we allow a dependence on $\alpha\in (0,\frac{1}{4})$ for the constant in \eqref{eq:Schauder_est}.
	For our purposes the dependence on $\alpha$ is not an issue since here
	and in the two appendices
	an application of the estimates for a fixed value of $\alpha\in (0,\frac{1}{4})$ suffices.
\end{remark}

Although Proposition \ref{prop:Schauder} may be well-known, we are not aware of a proof in the literature for the seminorm defined by \eqref{eq:hoelder_semi_norm}; hence we include a proof in Appendix \ref{section:proof_schauder}. Estimates on the half-space (which are related to those above but not sufficient for our application) are found already in the work of Solonnikov; cf.\ \cite[(3.1)-(3.6)]{S}.

We now turn to the application.

\begin{proof}[Proof of Lemma \ref{le:estu}]
	Fix $\alpha\in (0,\frac{1}{4})$.
	We will show the lower order estimate
	\begin{align}\label{eq:lower_order_terms}
		\norm{u}\lesssim \sup_{\T>0}\left(\left(\T^{\frac{1}{2}}\wedge\T^{\frac{3}{4}}\right)\norm{g}_{t\geq \T}\right)
	\end{align}
	and the higher order estimate
\begin{align}\label{eq:uxx_hoelder_est}
			\sup_{\tau>0}\T^{1+\alpha}[u_{xxxx},u_{xx},u_{t}]_{\alpha,t\geq 2\tau}
			\lesssim \sup_{\T>0}\left(\T^{1+\alpha}[g_{x}]_{\alpha,t\geq\T}+\T\norm{g_{x}}_{t\geq \T}+\T^{\alpha}[u]_{\alpha,t\geq\T}+\norm{u}_{t\geq\tau}\right).
		\end{align}	
Here and below, we use the notation
\begin{align}
  [f]_{\alpha,t\geq\tau}:=\sup_{t\geq \tau,x\geq 0}\sup_{s\geq 0,z\geq 0}
  \frac{\abs{f(t+s,x+z)-f(t,z)}}{(s+( z^2 \wedge z^4 ))^{\alpha}}.\label{taudef}
\end{align}
We begin by showing how \eqref{eq:lower_order_terms} and \eqref{eq:uxx_hoelder_est} lead to the result and then proceed to prove the estimates.

\step{Step 1:}
Using \eqref{eq:lower_order_terms}--\eqref{eq:uxx_hoelder_est} and interpolation (cf.\ Lemma~\ref{l:interp}), we deduce
\begin{align}
		\begin{split}\label{eq:intermediate-messy}
		\MoveEqLeft
		\sup_{\tau>0}\bigg(\T^{1+\alpha}[u_{xxxx},u_{xx},u_{t}]_{\alpha,t\geq 2\tau}+\T\norm{u_{xxxx},u_{xx},u_{t}}_{t\geq 2\tau}\\
		&\qquad \left.+\left(\T^{\frac{1}{2}}\vee\T^{\frac{1}{4}}\right)\norm{u_{x}}_{t\geq 2\tau}+\T^{\alpha}[u]_{\alpha,t\geq 2\tau}\right)\\
		&\lesssim \sup_{\T>0}\left(\T^{1+\alpha}[g_{x}]_{\alpha,t\geq\T}+\T\norm{g_{x}}_{t\geq \T}+\left(\T^{\frac{1}{2}}\wedge\T^{\frac{3}{4}}\right)\norm{g}_{t\geq\tau}\right).
		\end{split}
	\end{align}
Indeed, adding the ``missing'' left-hand side terms of \eqref{eq:intermediate-messy} to both sides of \eqref{eq:uxx_hoelder_est}, we use \eqref{eq:newinterpolate3}, \eqref{eq:newinterpolate2}, and \eqref{eq:newinterpolate1} for $\delta$ sufficiently small to the $u$-dependent right-hand side terms (other than $\norm{u}$). We absorb the resultant $\delta$-dependent terms into the left-hand side and then apply \eqref{eq:lower_order_terms} to control the infinity norm of $u$ in terms of $g$.

To pass from \eqref{eq:lower_order_terms} and \eqref{eq:intermediate-messy} to \eqref{eq:newtoshowu}, we use Lemma \ref{l:interp} together with the identities
	\begin{align}\label{ids} \T\left(\T^{\frac{1}{2}}\vee\T^{\frac{1}{4}}\right)=\left(\T^{\frac{1}{2}}\wedge\T^{\frac{3}{4}}\right)\left(\T\vee\T^{\frac{1}{2}}\right)\quad\text{and}\quad
		 \T=\left(\T^{\frac{1}{2}}\wedge\T^{\frac{3}{4}}\right)\left(\T^{\frac{1}{2}}\vee\T^{\frac{1}{4}}\right).
	\end{align}
Indeed, we observe on the one hand
\begin{align*}
  \T \norm{g_x}\overset{\eqref{ids}}=\left(\T^{\frac{1}{2}}\wedge\T^{\frac{3}{4}}\right)\left(\T^{\frac{1}{2}}\vee\T^{\frac{1}{4}}\right)\norm{g_x}
  \overset{\eqref{eq:newinterpolate2}}\lesssim\left(\T^{\frac{1}{2}}\wedge\T^{\frac{3}{4}}\right)\left(\T\norm{g_{xxxx},g_{xx}}+\norm{g_x}\right),
\end{align*}
and on the other hand
\begin{eqnarray*}
\T^{1+\alpha}[g_x]_{\alpha,t\geq\tau}&\overset{\eqref{eq:newinterpolate4}}\lesssim&\T
\left(\left(\T^\frac{1}{2}\wedge\T^\frac{3}{4}\right)\norm{g_t}+\left(\T^\frac{1}{2}\vee\T^\frac{1}{4}\right)
\norm{g_{xx}}+\norm{g_x}\right).
\end{eqnarray*}
For the second order term we estimate
\begin{align*}
\T\left(\T^\frac{1}{2}\vee\T^\frac{1}{4}\right)\norm{g_{xx}}
\overset{\eqref{ids}}=\left(\T^{\frac{1}{2}}\wedge\T^{\frac{3}{4}}\right)
\left(\T\vee\T^{\frac{1}{2}}\right)\norm{g_{xx}}
\overset{\eqref{eq:lasthingtoshow}}\lesssim \left(\T^{\frac{1}{2}}\wedge\T^{\frac{3}{4}}\right)\left(\T\norm{g_{xxxx},g_{xx}}+\norm{g}\right).
\end{align*}


	\step{Step 2:}
		To obtain the lower order estimate \eqref{eq:lower_order_terms}, let $g$ be extended by even reflection to the line. It is enough to derive an estimate on the even function $w$ satisfying
		\begin{align*}
			\left\{
			\begin{aligned}
				\U_{t}-\U_{xx}+\U_{xxxx}&=g		&&\text{on }	&&t\in (0,\infty),	&&x\in \R,\\
				\U&=0							&&\text{for }	&&t=0,				&& x\in \R,
			\end{aligned}
			\right.
		\end{align*}
since $w_x$ satisfies \eqref{eq:equ}.
		Using Duhamel's principle, we rewrite $\U$ as
		\begin{align*}
			\U=\int_{0}^{t}\U^{\T}\,\dd \T
			\qquad\text{for}\qquad
			\left\{
			\begin{aligned}
				\U^{\T}_{t}-\U_{xx}^{\T}+\U_{xxxx}^{\T}&=0	&&\text{on  }	&& t\in (\T,\infty), && x\in\R,\\
				\U^{\T}&=g									&&\text{for } && t=\T, && x\in \R.
			\end{aligned}
			\right.
		\end{align*}
		As in the proof of Lemma \ref{le:estv} (noting that now $\norm{g}$ instead of $\norm{g_{x}}$ is bounded), we observe that
		\begin{align}\notag
			\norm*{\U_{x}^{\T}(t,\cdot)}\lesssim \left(\frac{1}{(t-\T)^{\frac{1}{2}}}\wedge\frac{1}{(t-\T)^{\frac{1}{4}}}\right) \norm{g(\T,\cdot)},
		\end{align}
		which implies
		\begin{eqnarray*}
			\norm{\U_{x}(t,\cdot)}&\lesssim& \int_{0}^{t}\norm*{\U_{x}^{\T}(t,\cdot)}\,\dd \T
			\lesssim\int_0^t\left(\frac{1}{(t-\T)^{\frac{1}{2}}}
\wedge\frac{1}{(t-\T)^{\frac{1}{4}}}\right) \norm{g(\T,\cdot)}\dd \T\\
&=&\int_0^t\left(\frac{1}{(t-\T)^{\frac{1}{2}}}
\wedge\frac{1}{(t-\T)^{\frac{1}{4}}}\right)\left(\frac{1}{\T^\frac{1}{2}}\vee\frac{1}{\T^\frac{3}{4}}\right)
\left(\T^\frac{1}{2}\wedge\T^\frac{3}{4}\right)\norm{g(\T,\cdot)}\,\dd\T\\
			&\lesssim& \sup_{\T>0}\left(\T^{\frac{1}{2}}\wedge\T^{\frac{3}{4}}\right)\norm{g}_{t\geq\T},
		\end{eqnarray*}
		and hence \eqref{eq:lower_order_terms}.

	\step{Step 3:}
		To establish \eqref{eq:uxx_hoelder_est}, let $\hat{\eta}\colon\R\to[0,1]$ be a smooth cutoff function with
$\hat{\eta}(s)= 0$ for all $s\leq 1$ and $\hat\eta(s)=1$ for all $s\geq 2$ and define
$\eta\colon\R\to [0,1]$ via $\eta(t):=\hat\eta\left(\frac{t}{\tau}\right)$.
Then $\tilde u:=\eta u$  satisfies
		\begin{align*}
			\left\{
			\begin{aligned}
				\tilde u_{t}-\tilde u_{xx}+\tilde u_{xxxx}&=\dot\eta u+\eta g_{x}		&&\text{on }&&t\in \R,	&&x\in (0,\infty),\\
				\tilde u=\tilde{u}_{xx} &=0															 &&\text{for }
				&&t\in \R,	&& x=0,
			\end{aligned}
			\right.
		\end{align*}
		and the right-hand side of the differential equation is compactly supported backward in time.		
		According to Proposition \ref{prop:Schauder}, we have the estimate
		\begin{align*}
		[u_{xxxx},u_{xx},u_{t}]_{\alpha,t\geq 2\tau}
			&\leq	[\tilde u_{xxxx},\tilde u_{xx},\tilde u_{t}]_{\alpha}\lesssim 	[\dot\eta u+\eta g_{x}]_{\alpha}\\
			&\lesssim \norm{\dot\eta}[u]_{\alpha,t\geq \T}+[\dot\eta]_{\alpha}\norm{u}_{t\geq \T}+\norm{\eta}[g_{x}]_{\alpha,t\geq \T}+[\eta]_{\alpha}\norm{g_{x}}_{t\geq \T}\\
			&\lesssim \frac{1}{\T}[u]_{\alpha,t\geq \T}+\frac{1}{\T^{1+\alpha}}\norm{u}_{t\geq \T}+[g_{x}]_{\alpha,t\geq \T}+\frac{1}{\T^{\alpha}}\norm{g_{x}}_{t\geq \T},
		\end{align*}
		which implies \eqref{eq:uxx_hoelder_est}.
\end{proof}

\begin{proof}[Proof of Corollary \ref{co:estu}]
	We note that $tu_{t}-u$  satisfies
	\begin{align*}
		 \left(tu_{t}-u\right)_{t}-\left(tu_{t}-u\right)_{xx}+\left(tu_{t}-u\right)_{xxxx}=\left(tg_{t}+u_{x}-u_{xxx}\right)_{x}
	\end{align*}
	and
	\begin{align*}
		\left(tu_{t}-u\right)_{xx}=tu_{t}-u=0\quad\text{at }x=0.
	\end{align*}
	Combining \eqref{eq:newtoshowu} for $u$ and for $tu_{t}-u$ with the triangle inequality yields
	\begin{align*}
		\MoveEqLeft
		\sup_{\T>0}\left(\T\norm{t (u_{txxxx},u_{txx},u_{tt})}_{t\geq \T}+\T\norm{u_{xxxx},u_{xx},u_{t}}_{t\geq \T}+\norm{tu_{t},u}_{t\geq \T}\right)\\
		&\lesssim \sup_{\T>0}\left(\T^{\frac{1}{2}}\wedge\T^{\frac{3}{4}}\right)\left(\T\norm{t (g_{txxxx},g_{txx},g_{tt})}_{t\geq \T}
		+\norm{tg_{t}}_{t\geq \T}
		+\T\norm{g_{xxxx},g_{xx},g_{t}}_{t\geq \T}+\norm{g}_{t\geq \T}
		\right.\\
		&\qquad\qquad \left.+\T\norm{u_{7x},u_{5x},u_{txxx},u_{xxx},u_{tx}}_{t\geq \T}+\norm{u_{xxx},u_{x}}_{t\geq \T}\right).
	\end{align*}
	Using
$
		\sup_{\T>0}\T\norm{t\cdot}_{t\geq \T}=\sup_{\T>0}\T^{2}\norm{\cdot}_{t\geq \T}$
and
$\sup_{\T>0}\norm{t\cdot}_{t\geq \T}=\sup_{\T>0}\T\norm{\cdot}_{t\geq \T}$,
 this turns into
	\begin{align}
		\MoveEqLeft
		\sup_{\T>0}\left(\T^{2}\norm{u_{txxxx},u_{txx},u_{tt}}_{t\geq \T}+\T\norm{u_{xxxx},u_{xx},u_{t}}_{t\geq \T}+\norm{u}_{t\geq \T}\right)\notag\\
		&\lesssim \sup_{\T>0}\left(\T^{\frac{1}{2}}\wedge\T^{\frac{3}{4}}\right)\left(\T^{2}\norm{g_{txxxx},g_{txx},g_{tt}}_{t\geq \T}
		+\T\norm{g_{xxxx},g_{xx},g_{t}}_{t\geq \T}+\norm{g}_{t\geq \T}
		\right.\notag\\
		&\qquad\qquad \left.+\T\norm{u_{7x},u_{5x},u_{txxx},u_{xxx},u_{tx}}_{t\geq \T}+\norm{u_{xxx},u_{x}}_{t\geq \T}\right).\label{fstar}
	\end{align}
	On the one hand, applying $\partial_{x}^{4}$ to $-u_{xx}+u_{xxxx}=-u_t+g_{x}$ we get
	\begin{align*}
		\norm{u_{8x},u_{6x}}_{t\geq \tau}\lesssim \norm*{\left(\partial_{x}^{2}-1\right)u_{6x}}_{t\geq \tau}\lesssim \norm{u_{txxxx}}_{t\geq \tau}+\norm{g_{5x}}_{t\geq \tau}
	\end{align*}
	by elementary elliptic regularity for $(-\partial_x^2+1)$. (See for instance the details below \eqref{eq:1_b_alpha-beta} in the appendix.) On the other hand applying $\partial_{x}^{2}$ to $-u_{xx}+u_{xxxx}=-u_t+g_{x}$ yields
	\begin{align*}
		\norm{u_{6x},u_{xxxx}}_{t\geq \tau}\lesssim \norm*{\left(\partial_{x}^{2}-1\right)u_{xxxx}}_{t\geq \tau}\lesssim \norm{u_{txx}}_{t\geq \tau}+\norm{g_{xxx}}_{t\geq \tau}.
	\end{align*}
	Hence
	\begin{align}\label{eq:nb4}
		\norm{u_{8x},u_{6x},u_{xxxx}}_{t\geq \tau}\lesssim \norm{u_{txxxx},u_{txx}}_{t\geq \tau}+\norm{g_{5x},g_{xxx}}_{t\geq \tau}.
	\end{align}
	By elementary spatial interpolation (applying \eqref{eq:newinterpolate2} to the functions $g_{xx}$ and $g_{xxxx}$) we estimate
	\begin{align*}
		\T^{2}\norm{g_{5x},g_{xxx}}_{t\geq \tau}\lesssim \left(\T^{\frac{1}{2}}\wedge\T^{\frac{3}{4}}\right)\left(\T^{2}\norm{g_{8x},g_{6x},g_{xxxx}}_{t\geq \tau}+\T\norm{g_{xxxx},g_{xx}}_{t\geq \tau}
\right).
	\end{align*}
	Together with \eqref{eq:nb4} this gives
	\begin{align*}
		\lefteqn{\T^{2}\norm{u_{8x},u_{6x},u_{xxxx}}_{t\geq \tau}}\\
&\lesssim \T^{2}\norm{u_{txxxx},u_{txx}}_{t\geq \tau}+\left(\T^{\frac{1}{2}}\wedge\T^{\frac{3}{4}}\right)\left(\T^{2}\norm{g_{8x},g_{6x},g_{xxxx}}_{t\geq \tau}+\T\norm{g_{xxxx},g_{xx}}_{t\geq \tau}
\right).
	\end{align*}
Combining this with \eqref{fstar} yields
	\begin{align}
		\MoveEqLeft
		\sup_{\T>0}\left(\T^{2}\norm{u_{8x},u_{6x},u_{txxxx},u_{xxxx},u_{txx},u_{tt}}_{t\geq \tau}+\T\norm{u_{xxxx},u_{xx},u_{t}}_{t\geq \T}+\norm{u}_{t\geq \tau}\right)\notag\\
		&\lesssim \sup_{\T>0}\left(\T^{\frac{1}{2}}\wedge\T^{\frac{3}{4}}\right)\left(\T^{2}\norm{g_{8x},g_{6x},g_{txxxx},g_{xxxx},g_{txx},g_{tt}}_{t\geq \tau}
		+\T\norm{g_{xxxx},g_{xx},g_{t}}_{t\geq \tau}+\norm{g}_{t\geq \tau}
		\right.\notag\\
		&\qquad\qquad \left.+\T\norm{u_{7x},u_{5x},u_{txxx},u_{xxx},u_{tx}}_{t\geq \tau}+\norm{u_{xxx},u_{x}}_{t\geq \tau}\right).\label{andthus}
	\end{align}
	By elementary interpolation (similar to the proof of \eqref{eq:newinterpolate2}) and Young's inequality, we have for any $\delta\in(0,1]$ that
	\begin{align}
		\MoveEqLeft \sup_{\T>0}\left(\T^{\frac{1}{2}}\wedge\T^{\frac{3}{4}}\right)\left(\T\norm{u_{7x},u_{5x},u_{txxx},u_{xxx},u_{tx}}_{t\geq \tau}+\norm{u_{xxx},u_{x}}_{t\geq \tau}\right)\notag\\
		&\lesssim \delta \T^{2}\norm{u_{8x},u_{6x},u_{txxxx},u_{xxxx},u_{txx},u_{tt}}_{t\geq \tau} + \frac{1}{\delta^{3}} \left(\T\norm{u_{xxxx},u_{xx},u_{t}}_{t\geq \T}+\norm{u}_{t\geq \tau}\right).\label{stst}
	\end{align}
For example, we employ estimates such as
\begin{align*}
  \left(\T^\half\wedge\T^\thfr\right)\T\norm{u_{tx}}&\lesssim\T^\frac{3}{2}\norm{u_{txx}}^\half\norm{u_t}^\half
  =\left(\T^2\norm{u_{txx}}\right)^\half\left(\T\norm{u_t}\right)^\half,\\
  \left(\T^\fr12\wedge\T^\fr34\right)\T\norm{u_{xxx}}& \lesssim\T^\fr32\norm{u_{xxxx}}^\fr12\norm{u_{xx}}^\fr12=\left(\T^2\norm{u_{xxxx}}\right)^\fr12\left(\T\norm{u_{xx}}\right)^\fr12.
\end{align*}

Choosing $\delta$ sufficiently small to absorb the first summand on the right-hand side of \eqref{stst} and controlling the second summand  by \eqref{eq:newtoshowu}, one obtains from \eqref{andthus} the estimate \eqref{eq:estu}.
\end{proof}

\begin{appendix}

\section{Interpolation estimates}\label{section:interpolationproof}

In this section we show the interpolation inequalities that were essential in Subsection \ref{subsection:parabolicvariable}. We remind the reader of the definitions \eqref{infdef}, \eqref{taudef}.
\begin{notation}
  For notational simplicity we will use $\norm{\cdot}$ to denote $\norm{\cdot}_\infty$.
\end{notation}
\begin{lemma}\label{l:interp}
For a smooth function $u:[0,\infty)^2\to\R$ and parameters $\alpha\in(0, \frac{1}{4})$, $\delta\in(0,1]$, $\tau\in (0,\infty)$, there holds
\begin{align}
	\T\norm{u_{xxxx},u_{xx},u_{t}}_{t\geq \tau}&\lesssim  \delta^{\alpha}\T^{1+\alpha}[u_{xxxx},u_{xx},u_{t}]_{\alpha,t\geq\tau}
+\frac{1}{\delta}\norm{u}_{t\geq\tau},\label{eq:newinterpolate1}\\
	\left(\T^{\frac{1}{2}}\vee \T^{\frac{1}{4}}\right)\norm{u_{x}}_{t\geq\tau}
	&\lesssim \delta\T \norm{u_{xxxx},u_{xx}}_{t\geq\tau}+\frac{1}{\delta}\norm{u}_{t\geq\tau},\label{eq:newinterpolate2}\\
		\left(\T\vee\T^{\frac{1}{2}}\right)\norm{u_{xx}}_{t\geq\tau}
		&\lesssim \delta\T\norm{u_{xxxx},u_{xx}}_{t\geq\tau}+\frac{1}{\delta}\norm{u}_{t\geq\tau},\label{eq:lasthingtoshow}\\
	\T^{\alpha}[u]_{\alpha,t\geq \tau}&\lesssim \delta\T\norm{u_{t}}_{t\geq\tau}+\delta^{\frac{1-4\alpha}{4\alpha}}\left(\T^{\frac{1}{2}}\vee \T^{\frac{1}{4}}\right)\norm{u_{x}}_{t\geq\tau}
+\frac{1}{\delta}\norm{u}_{t\geq\tau},\label{eq:newinterpolate3}\\
	\T^{\alpha}[u_{x}]_{\alpha,t\geq \tau}&\lesssim \left(\T^{\frac{1}{2}}\wedge\T^{\frac{3}{4}}\right)\norm{u_{t}}_{t\geq\tau}
+\left(\T^{\frac{1}{2}}\vee\T^{\frac{1}{4}}\right)\norm{u_{xx}}_{t\geq\tau}
+\norm{u_{x}}_{t\geq\tau}.\label{eq:newinterpolate4}
\end{align}
\end{lemma}
\begin{proof}
	To deduce \eqref{eq:newinterpolate1}, we write $u$ as its forward (anticipating) temporal convolution on scale $\delta\T$ plus the difference:
$u=\eta\ast u+(u-\eta\ast u).$	Differentiating and using the triangle inequality, we obtain for $t\geq\tau$ that
	\begin{align}
		\begin{split}\label{eq:conv1}
			\MoveEqLeft
			\abs{u_{t}}\leq \abs{\eta_{t}\ast u}+\abs{u_{t}-\eta\ast u_{t}}
			\lesssim \frac{1}{\delta\T}\norm{u} +\int_0^\infty \abs{u_{t}(t+s,x)-u_{t}(t,x)}\abs{\eta(s)}\,\dd s\\
			&\lesssim \frac{1}{\delta\T}\norm{u} +[u_{t}]_{\alpha,t\geq\tau} \int_0^\infty \abs{s}^{\alpha}\abs{\eta(s)}\,\dd s
			\lesssim \frac{1}{\delta\T}\norm{u} +(\delta\T)^{\alpha}[u_{t}]_{\alpha,t\geq\tau}.
		\end{split}
	\end{align}
	Similarly, forward spatial convolution on scale $h$ yields for $t\geq\tau$ that
	\begin{align}
		\begin{split}\label{eq:conv2}
		\MoveEqLeft
		\abs{u_{xx}}\leq \abs{\eta_{xx}\ast u}+\abs{u_{xx}-\eta\ast u_{xx}}
		\lesssim \frac{1}{h^{2}}\norm{u} +\int_0^\infty \abs{u_{xx}(t,x)-u_{xx}(t,x+z)}\abs{\eta(z)}\,\dd z\\
		&\lesssim \frac{1}{h^{2}}\norm{u} +[u_{xx}]_{\alpha,t\geq\tau} \int_0^\infty \left( z^2 \wedge z^4 \right)^{\alpha}\abs{\eta(z)}\,\dd z
		\lesssim \frac{1}{h^{2}}\norm{u} +\left(\abs{h}^{2}\wedge\abs{h}^{4}\right)^{\alpha}[u_{xx}]_{\alpha,t\geq\tau}
		\end{split}
	\end{align}
	and
	\begin{align}\label{eq:conv3}
		\abs{u_{xxxx}}\lesssim\frac{1}{h^{4}}\norm{u} +\left(\abs{h}^{2}\wedge\abs{h}^{4}\right)^{\alpha}[u_{xxxx}]_{\alpha,t\geq\tau}.
	\end{align}
	Combining \eqref{eq:conv1}, \eqref{eq:conv2} for $h=(\delta\T)^{\frac{1}{2}}$, and \eqref{eq:conv3} for $h=(\delta\T)^{\frac{1}{4}}$ establishes \eqref{eq:newinterpolate1}.

To show \eqref{eq:newinterpolate2} for $\T\geq 1$ we note that by classical interpolation (cf.\ \cite[Theorem 3.2.1]{K}), we have
\begin{align*}
	 \T^{\frac{1}{2}}\norm{u_{x}}\lesssim\T^{\frac{1}{2}}\norm{u_{xx}}^{\frac{1}{2}}\norm{u}^{\frac{1}{2}}\lesssim \delta\T\norm{u_{xx}}+\frac{1}{\delta}\norm{u}.
\end{align*}
For $\T<1$, we use
\begin{align*}
	 \T^{\frac{1}{4}}\norm{u_{x}}\lesssim\T^{\frac{1}{4}}\norm{u_{xxxx}}^{\frac{1}{4}}\norm{u}^{\frac{3}{4}}\lesssim \delta\T\norm{u_{xxxx}}+\frac{1}{\delta^{\frac{1}{3}}}\norm{u}.
\end{align*}
The proof of \eqref{eq:lasthingtoshow} is similar and we omit it.

For \eqref{eq:newinterpolate3}, we first consider for $t\geq \tau$ and $s\geq 0$ the temporal estimate
\begin{align*}
	\T^{\alpha}\abs{u(t+s,x+z)-u(t,x+z)}\lesssim
	\left(\delta\T\norm{u_{t}}_{t\geq\T}+\frac{1}{\delta}\norm{u}_{t\geq\T}\right)\abs{s}^{\alpha}
\end{align*}
which follows from
\begin{align*}
	\abs{u(t+s,x+z)-u(t,x+z)}\lesssim
		\Big(\norm{u}_{t\geq\T}\wedge
		\big(\norm{u_{t}}_{t\geq\T}\abs{s}\big)\Big)
	\lesssim \norm{u}_{t\geq\T}^{1-\alpha}\norm{u_{t}}_{t\geq\T}^{\alpha}\abs{s}^{\alpha}.
\end{align*}
The spatial part is estimated by
\begin{align*}
	\abs{u(t,x+z)-u(t,x)}\lesssim \left(\delta\left(\T^{\frac{1}{2}}\vee \T^{\frac{1}{4}}\right)\norm{u_{x}}_{t\geq\T}+\frac{1}{\delta}\norm{u}_{t\geq\T}\right)\left( z^2 \wedge z^4 \right)^{\alpha}.
\end{align*}
To see this, we use
\begin{align*}
	\abs{u(t,x+z)-u(t,x)}\lesssim
	\Big(
		\norm{u}_{t\geq\T}\wedge
		\big(\norm{u_{x}}_{t\geq\T}\abs{z}\big)
	\Big)
\end{align*}
and distinguish cases. For $\abs{z}\leq 1$ we obtain
\begin{align*}
	\T^{\alpha}\abs{u(t,x+z)-u(t,x)}\lesssim \norm{u}_{t\geq\T}^{1-4\alpha}\left(\T^{\frac{1}{4}}\norm{u_{x}}_{t\geq\T}\right)^{4\alpha}\abs{z}^{4\alpha}
	\lesssim \left(\frac{1}{\delta}\norm{u}_{t\geq\T}+\delta^{\frac{1-4\alpha}{4\alpha}}\T^{\frac{1}{4}}\norm{u_{x}}_{t\geq\T}\right)
\abs{z}^{4\alpha}.
\end{align*}
The analogous estimate for  $\abs{z}> 1$ gives
\begin{align*}
	\T^{\alpha}\abs{u(t,x+z)-u(t,x)}	\lesssim \left(\frac{1}{\delta}\norm{u}_{t\geq\T}+\delta^{\frac{1-2\alpha}{2\alpha}}\T^{\frac{1}{2}}
\norm{u_{x}}_{t\geq\T}\right)\abs{z}^{2\alpha}.
\end{align*}


For \eqref{eq:newinterpolate4} we note that
using forward spatial convolution on scale $h$ gives
\begin{align*}
	\MoveEqLeft
	\abs{u_{x}(t,x)-u_{x}(t+s,x)}\lesssim \norm{u_{x}-\eta\ast u_{x}}_{t\geq\T}+\abs*{\eta_{x}\ast\int_{t}^{t+s}u_{t}(\xi,x)\,\dd \xi}\\
	&\lesssim h\norm{u_{xx}}_{t\geq\T}+\frac{1}{h}\norm{u_{t}}_{t\geq\T}\,\abs{s}.
\end{align*}
Optimization in $h$ yields
\begin{align*}
	\abs{u_{x}(t,x)-u_{x}(t+s,x)}\lesssim \left(\norm{u_{t}}_{t\geq\T}\,\norm{u_{xx}}_{t\geq\T}\,\abs{s}\right)^{\frac{1}{2}}
\end{align*}
and hence
\begin{align*}
	\abs{u_{x}(t,x)-u_{x}(t+s,x)}\lesssim \norm{u_{t}}_{t\geq\T}^{\alpha}\,\norm{u_{xx}}_{t\geq\T}^{\alpha}\,\norm{u_{x}}_{t\geq\T}^{1-2\alpha}\,\abs{s}^{\alpha}.
\end{align*}
From here we deduce
\begin{align*}
	\T^{\alpha}\frac{\abs{u_{x}(t,x)-u_{x}(t+s,x)}}{\abs{s}^{\alpha}}
	&\lesssim
	\T^{\alpha}\norm{u_{t}}_{t\geq\T}^{\alpha}\norm{u_{xx}}_{t\geq\T}^\alpha\norm{u_{x}}_{t\geq\T}^{1-2\alpha}\\
&\lesssim \T^{\frac{1}{2}}\,\norm{u_{t}}_{t\geq\T}^{\frac{1}{2}}\,\norm{u_{xx}}_{t\geq\T}^\frac{1}{2}+\norm{u_{x}}_{t\geq\T}\\
&\lesssim (\T^{\frac{1}{2}}\wedge\T^\frac{3}{4})\norm{u_{t}}_{t\geq\T}
+\frac{\T}{\T^{\frac{1}{2}}\wedge\T^\frac{3}{4}}\norm{u_{xx}}_{t\geq\T}+\norm{u_{x}}_{t\geq\T}\\
&=(\T^{\frac{1}{2}}\wedge\T^\frac{3}{4})\norm{u_{t}}_{t\geq\T}
+(\T^\frac{1}{2}\vee \T^\frac{1}{4})\norm{u_{xx}}_{t\geq\T}+\norm{u_{x}}_{t\geq\T}.
\end{align*}
For the spatial part, we obtain
\begin{align*}
	\T^{\alpha}\abs{u_{x}(t,x)-u_{x}(t,x+z)}
	&\lesssim
	\left.
	\begin{cases}
		 \left(\T^\frac{1}{2}\norm{u_{xx}}_{t\geq\T}\right)^{2\alpha}\norm{u_{x}}_{t\geq\T}^{1-2\alpha}\abs{z}^{2\alpha}\\
		 \left(\T^\frac{1}{4}\norm{u_{xx}}_{t\geq\T}\right)^{4\alpha}\norm{u_{x}}_{t\geq\T}^{1-4\alpha}\abs{z}^{4\alpha}
	\end{cases}
	\right\}\\
	&\lesssim \left(\left(\T^{\frac{1}{2}}\vee\T^{\frac{1}{4}}\right)\norm{u_{xx}}_{t\geq\T}+\norm{u_{x}}_{t\geq\T}\right)\left( z^2 \wedge z^4 \right)^{\alpha}.
\end{align*}
\end{proof}

\section{Proof of Proposition \ref{prop:Schauder}}\label{section:proof_schauder}
In this section we will prove Proposition \ref{prop:Schauder}.
 As usual, Schauder theory for a half space can be split into H\"older estimates on the whole space (Step~1) and H\"older estimates for inhomogeneous boundary data but vanishing right-hand side (Step~3).
Both estimates will rely on representation by a kernel: the (translation invariant) heat kernel and the Poisson kernel, respectively. The latter can be explicitly recovered from the heat kernel in the usual way. Here, we will use that the heat kernel behaves like the usual (second order) heat kernel for large times, but like the ``biharmonic heat kernel'' for small times. Both scaling behaviors are important in order to catch the crossover in the Carnot-Carath\'eodory distance entering the H\"older norm (cf.~\eqref{eq:hoelder_semi_norm}). All this will be used first to get maximal regularity on the level of the time derivative. In order to capture $u_{xxxx}$ and $u_{xx}$ individually, this information is then postprocessed by (one-dimensional and thus elementary) elliptic theory.

Splitting $f$ as
\begin{align*}
	f=(f-f(x=0))+f(x=0),
\end{align*}
we observe that it is sufficient to show \eqref{eq:Schauder_est} for solutions of
\begin{align}
	\left\{
	\begin{aligned}
		u_{t}-u_{xx}+u_{xxxx}&=f		&&\text{on }	&&t\in \R,	&&x\in (0,\infty),\\
		u=u_{xx}&=0					&&\text{for }	&&t\in \R,	&& x=0,
	\end{aligned}
	\right.
\end{align}
for the two cases
\begin{align*}
\text{(i) $f(t,x=0)=0$}\quad\text{ and }\quad\text{(ii) $f(x,t)=f(t)$.}
\end{align*}

 For the first case, an odd reflection leads to a well-behaved problem on the line.
In the second case, we take a primitive $F$ of $f$ and consider  $w:=F-u$, which satisfies
\begin{align*}
	\left\{
	\begin{aligned}
		w_{t}-w_{xx}+w_{xxxx}&=0		&&\text{on }&&t\in\R, &&x\in(0,\infty),\\
		w&=F							&&\text{for }&&t\in\R, &&x=0,\\
		w_{xx}&=0						&&\text{for }&&t\in\R, &&x=0.
	\end{aligned}
	\right.
\end{align*}
We will now develop the necessary estimates for $u$ in case (i) and $w$ in case (ii).

\begin{notation}
  For notational simplicity we will use $\norm{\cdot}$ to denote $\norm{\cdot}_\infty$.
\end{notation}

\step{Step~0: Heat kernel estimates}

We recall from the proof of Proposition~\ref{paraboliczeta} that the heat kernel of $\partial_t-\partial_x^2-\partial_x^4$ is of the form
\begin{align*}
	H(t,x)=H^{(2)}(t,x)\ast H^{(4)}(t,x)
\end{align*}
where
\begin{align}\label{B2}
	 H^{(2)}(t,x)=\frac{1}{t^{\frac{1}{2}}}H^{(2)}\left(1,\frac{x}{t^{\frac{1}{2}}}\right)\qquad\text{and}\qquad
	H^{(4)}(t,x)=\frac{1}{t^{\frac{1}{4}}}H^{(4)}\left(1,\frac{x}{t^{\frac{1}{4}}}\right)
\end{align}
for two Schwartz functions $H^{(2)}(1,\cdot)$ and $H^{(4)}(1,\cdot)$ with integral $1$ and vanishing first moment.

For $\alpha\in [0,\frac{1}{4}]$ we will show the estimates
\begin{align}
	\int\abs{H_{t}(t,x)}\left(\abs{x}^{4}\wedge\abs{x}^{2}\right)^{\alpha}\,\dd x&\lesssim \frac{1}{t^{1-\alpha}},\label{eq:to_show_heatkernel_1}\\
	\int\abs{H_{tt}(t,x)}\left(\abs{x}^{4}\wedge\abs{x}^{2}\right)^{\alpha}\,\dd x&\lesssim \frac{1}{t^{2-\alpha}},\label{eq:to_show_heatkernel_2}\\
	\int\abs{H_{tx}(t,x)}\left(\abs{x}^{4}\wedge\abs{x}^{2}\right)^{\alpha}\,\dd x&\lesssim \frac{1}{t^{1+\frac{1}{4}-\alpha}}\wedge\frac{1}{t^{1+\frac{1}{2}-\alpha}}.\label{eq:to_show_heatkernel_3}
\end{align}
Using the equation for $H$, we see that it is enough to prove
\begin{align}
	\int \abs{\partial_{x}^{k}H(t,x)}\abs{x}^{\beta}\,\dd x&\lesssim \left(t^{\frac{1}{4}}\right)^{\beta-k}\qquad\text{for}\quad t\leq 1,\label{eq:0f}\\
	\int \abs{\partial_{x}^{k}H(t,x)}\abs{x}^{\beta}\,\dd x&\lesssim \left(t^{\frac{1}{2}}\right)^{\beta-k}\qquad\text{for}\quad t\geq 1,\label{eq:0g}
\end{align}
for $k= 0,1,\ldots,8$ and $\beta\geq 0$.
To show these estimates, we will use the triangle inequality together with estimates for $H^{(2)}$, $H^{(4)}$ and
\begin{align}
	(H-H^{(4)})\left(t,t^{\frac{1}{4}}\hat x\right)
	&=t^{\frac{1}{4}}\int_{0}^{1}(1-\theta)\int H^{(4)}_{\hat x\hat x}\left(1,\hat x-\theta t^{\frac{1}{4}}\hat y\right)H^{(2)}(1,\hat y)\hat y^{2}\,\dd \hat y\,\dd \theta,\label{eq:0d}\\
	(H-H^{(2)})\left(t,t^{\frac{1}{2}}\hat x\right)
	&=\frac{1}{t}\int_{0}^{1}(1-\theta)\int H_{2\hat x\hat x}\left(1,\hat x-\theta \frac{1}{t^{\frac{1}{4}}}\hat y\right)H^{(4)}(1,\hat y)\hat y^{2}\,\dd \hat y\,\dd \theta\label{eq:0e}.
\end{align}

We first consider \eqref{eq:0d} and use
\begin{align*}
	\int H^{(2)}(t,y)\,\dd y=1\qquad\text{and}\qquad
	\int y H^{(2)}(t,y)\,\dd y=0
\end{align*}
to rewrite
\begin{align*}
	(H-H^{(4)})(t,x)=\int \bigl(H^{(4)}(t,x-y)-H^{(4)}(t,x)+H^{(4)}_x(t,x)y\bigr)H^{(2)}(t,y)\,\dd y.
\end{align*}
Changing variables via $x=t^{\frac{1}{4}}\hat x$ and $y=t^{\frac{1}{2}}\hat y$, we obtain from \eqref{B2} the form
\begin{align*}
	\MoveEqLeft
	(H-H^{(4)})\left(t,t^{\frac{1}{4}}\hat x\right)
	=\frac{1}{t^{\frac{1}{4}}}\int \left(H^{(4)}\left(1,\hat x-t^{\frac{1}{4}}\hat y\right)-H^{(4)}(1,\hat x)+H^{(4)}_x(1,\hat x)t^{\frac{1}{4}}\hat y\right)H^{(2)}(1,\hat y)\,\dd \hat y\\
	&=\frac{1}{t^{\frac{1}{4}}}\int \left(t^{\frac{1}{4}}\hat y\right)^{2}\int_{0}^{1} (1-\theta)H^{(4)}_{\hat x\hat x}\left(1,\hat x-\theta t^{\frac{1}{4}}\hat y\right)\,\dd \theta H^{(2)}(1,\hat y)\,\dd \hat y,
\end{align*}
which is \eqref{eq:0d}.
For \eqref{eq:0e} we use a very similar argument.

To argue for \eqref{eq:0f}, we differentiate \eqref{eq:0d} $k$ times with respect to $\hat{x}$ and obtain
\begin{align}
	\begin{split}\label{B10}
		\MoveEqLeft
		\left(t^{\frac{1}{4}}\right)^{k-\beta}\partial_{x}^{k}(H-H^{(4)})\left(t,t^{\frac{1}{4}}\hat x\right)\abs*{t^{\frac{1}{4}}\hat x}^{\beta}\\
		&=t^{\frac{1}{4}}\int_{0}^{1}(1-\theta)\int \partial_{\hat x}^{k+2}H^{(4)}\left(1,\hat x-\theta t^{\frac{1}{4}}\hat y\right)\abs*{\hat x}^{\beta}H^{(2)}(1,\hat y)\hat y^{2}\,\dd \hat y\,\dd \theta.
	\end{split}
\end{align}
Using
\begin{align*}
	\MoveEqLeft
	\int \abs*{\partial_{\hat x}^{k+2}H^{(4)}\left(1,\hat x-\theta t^{\frac{1}{4}}\hat y\right)}\abs*{\hat x}^{\beta}\,\dd \hat x\\
	&\lesssim \int \abs*{\partial_{\hat x}^{k+2}H^{(4)}\left(1,\hat x-\theta t^{\frac{1}{4}}\hat y\right)}\abs*{\hat x-\theta t^{\frac{1}{4}}\hat y}^{\beta}\,\dd \hat x
	+\abs*{\theta t^{\frac{1}{4}}\hat y}^{\beta}\int \abs*{\partial_{\hat x}^{k+2}H^{(4)}\left(1,\hat x-\theta t^{\frac{1}{4}}\hat y\right)}\,\dd \hat x\\
	&=\int \abs*{\partial_{\hat x}^{k+2}H^{(4)}\left(1,\hat x\right)}\abs*{\hat x}^{\beta}\,\dd \hat x
	+\abs*{\theta t^{\frac{1}{4}}\hat y}^{\beta}\int \abs*{\partial_{\hat x}^{k+2}H^{(4)}\left(1,\hat x\right)}\,\dd \hat x\lesssim 1+\abs*{\theta t^{\frac{1}{4}}\hat y}^{\beta},
\end{align*}
we infer from \eqref{B10} that
\begin{align*}
	\MoveEqLeft
	 \left(t^{\frac{1}{4}}\right)^{k-\beta}\int\abs*{\partial_{x}^{k}(H-H^{(4)})\left(t,x\right)}\abs*{x}^{\beta}\,\dd x\\
	&\lesssim t^{\frac{1}{2}}\int\int_{0}^{1}(1-\theta)\int \abs*{\partial_{\hat x}^{k+2}H^{(4)}\left(1,\hat x-\theta t^{\frac{1}{4}}\hat y\right)}\abs*{\hat x}^{\beta}
	\abs*{H^{(2)}(1,\hat y)}\hat y^{2}\,\dd \hat y\,\dd \theta\,\dd \hat x\\
	&\lesssim t^{\frac{1}{2}}\int\abs*{H^{(2)}(1,\hat y)}\hat y^{2}\,\dd \hat y+t^{\frac{1}{2}+\frac{\beta}{4}}\int\abs*{H^{(2)}(1,\hat y)}\abs{\hat y}^{2+\beta}\,\dd \hat y
	\lesssim t^{\frac{1}{2}}+t^{\frac{1}{2}+\frac{\beta}{4}}
	\overset{t\leq 1}\lesssim t^{\frac{1}{2}}.
\end{align*}
Finally, we obtain \eqref{eq:0f} by using the previous estimate and the triangle inequality
\begin{align*}
	\MoveEqLeft
	\int\abs*{\partial_{x}^{k}H\left(t,x\right)}\abs*{x}^{\beta}\,\dd x
	\leq \int\abs*{\partial_{x}^{k}H^{(4)}\left(t,x\right)}\abs*{x}^{\beta}\,\dd x+\int\abs*{\partial_{x}^{k}(H-H^{(4)})\left(t,x\right)}\abs*{x}^{\beta}\,\dd x\\
	&\overset{\eqref{B2}}\lesssim \left(t^{\frac{1}{4}}\right)^{\beta-k}+\left(t^{\frac{1}{4}}\right)^{\beta-k+2}
	\overset{t\leq 1}\lesssim \left(t^{\frac{1}{4}}\right)^{\beta-k}.
\end{align*}
For \eqref{eq:0g} we use a similar argument based on the representation \eqref{eq:0e}.

\step{Step 1: The problem on the line: Estimates for case (i)}

On the line, we split \eqref{eq:Schauder_est} into
\begin{align}\label{eq:Schauder_est_whole_space}
	[u_{t}-f]_{\alpha}\lesssim [f]_{\alpha} \qquad\text{and}\qquad
	[u_{xx}]_{\alpha}+[u_{xxxx}]_{\alpha}\lesssim [u_{t}-f]_{\alpha}.
\end{align}
We further split the first item in \eqref{eq:Schauder_est_whole_space} into
\begin{align}
	\abs{(u_{t}-f)(t,0)-(u_{t}-f)(0,0)}&\lesssim t^{\alpha}[f]_{\alpha}\qquad &&\text{for } t\geq 0,\label{1_a_alpha}\\
	\abs{(u_{t}-f)(0,x)-(u_{t}-f)(0,0)}&\lesssim x^{2\alpha}[f]_{\alpha}\qquad&&\text{for } x\geq 1,\label{1_a_beta}\\
	\abs{(u_{t}-f)(0,x)-(u_{t}-f)(0,0)}&\lesssim x^{4\alpha}[f]_{\alpha}\qquad&&\text{for } 0\leq x\leq 1,\label{1_a_gamma}
\end{align}
where by translation invariance and for notational convenience, we may focus on the argument $(0,0)$, and on positive increments, i.e., $t,x\geq 0$.
Using the representation
\begin{align}\label{eq:representation}
	u(t,x)=\int_{-\infty}^{t}\int H(t-s,x-y)f(s,y)\,\dd y\,\dd s\qquad\text{and}\qquad
	\int H(t-s,y)\,\dd y=0,
\end{align}
we obtain
\begin{align*}
	(u_{t}-f)(t,x)=\int_{-\infty}^{t}\int H_{t}(t-s,x-y)\bigl( f(s,y)-f(s,x) \bigr)\,\dd y\,\dd s.
\end{align*}
To show \eqref{1_a_alpha}, we note
\begin{align*}
	\abs{(u_{t}-f)(t,0)-(u_{t}-f)(0,0)}
	&\leq\int_{0}^{t}\int \abs*{H_{t}(t-s,-y)}\abs*{ f(s,y)-f(s,0) }\,\dd y\,\dd s\\
	& +\int_{-\infty}^{0}\int \abs*{H_{t}(t-s,-y)-H_{t}(-s,-y)}\abs*{ f(s,y)-f(s,0) }\,\dd y\,\dd s.
\end{align*}
For the first integral, we note
\begin{align*}
	\MoveEqLeft\int_{0}^{t}\int \abs*{H_{t}(t-s,-y)}\abs*{ f(s,y)-f(s,0) }\,\dd y\,\dd s
	\overset{\eqref{eq:hoelder_semi_norm},\eqref{eq:to_show_heatkernel_1}}\lesssim [f]_{\alpha}\int_{0}^{t}\frac{1}{(t-s)^{1-\alpha}}\,\dd s
	\lesssim[f]_{\alpha}t^{\alpha}.
\end{align*}
The second integral is estimated as
\begin{eqnarray*}
	\lefteqn{
	\int_{-\infty}^{0}\int \abs*{H_{t}(t-s,-y)-H_{t}(-s,-y)}\abs*{ f(s,y)-f(s,0) }\,\dd y\,\dd s}\\
	&\overset{\eqref{eq:hoelder_semi_norm}}\lesssim& [f]_{\alpha}\int_{-\infty}^{0}\int \int_{-s}^{t-s} \abs*{H_{tt}(\sigma,-y)}\left( y^4 \wedge y^2 \right)^{\alpha}\,\dd \sigma\,\dd y\,\dd s\\
	&\overset{\eqref{eq:to_show_heatkernel_2}}\lesssim& [f]_{\alpha}\int_{-\infty}^{0} \int_{-s}^{t-s} \frac{1}{\sigma^{2-\alpha}}\,\dd \sigma\,\dd s
	\lesssim [f]_{\alpha}t^{\alpha}.
\end{eqnarray*}
This establishes \eqref{1_a_alpha}.

For \eqref{1_a_beta}, we rewrite
\begin{align*}
	\MoveEqLeft
	(u_{t}-f)(0,x)-(u_{t}-f)(0,0)\\
	&\overset{\eqref{eq:representation}}=\int_{-x^{2}}^{0}\int H_{t}(-s,x-y)\bigl(f(s,y)-f(s,x)\bigr)\,\dd y\,\dd s-\int_{-x^{2}}^{0}\int H_{t}(-s,-y)\bigl(f(s,y)-f(s,0)\bigr)\,\dd y\,\dd s\\
	&\qquad +\int_{-\infty}^{-x^{2}}\int \left(H_{t}(-s,x-y)-H_{t}(-s,-y)\right)f(s,y)\,\dd y\,\dd s.
\end{align*}
The two right-hand side lines are estimated via
\begin{align*}
&	\abs*{\int_{-x^{2}}^{0}\int H_{t}(-s,x-y)\bigl(f(s,y)-f(s,x)\bigr)\,\dd y\,\dd s}
+\abs*{\int_{-x^{2}}^{0}\int H_{t}(-s,-y)\bigl(f(s,y)-f(s,0)\bigr)\,\dd y\,\dd s}\\
	&\qquad\overset{\eqref{eq:hoelder_semi_norm}, \eqref{eq:to_show_heatkernel_1}}\lesssim [f]_{\alpha}\int_{-x^{2}}^{0}\frac{1}{(-s)^{1-\alpha}}\,\dd s
	\lesssim [f]_{\alpha}{x}^{2\alpha}
\end{align*}
and
\begin{eqnarray*}
	\lefteqn{
	\abs*{\int_{-\infty}^{-x^{2}}\int \left(H_{t}(-s,x-y)-H_{t}(-s,-y)\right)f(s,y)\,\dd y\,\dd s}}\\
	&=&\abs*{\int_{-\infty}^{-x^{2}}\int x\int_{0}^{1}H_{tx}(-s,\theta x-y)\,\dd \theta f(s,y)\,\dd y\,\dd s}\\
	&=&\abs*{x\int_{-\infty}^{-x^{2}}\int_{0}^{1}\int H_{tx}(-s,\theta x-y) \bigl(f(s,y)-f(s,\theta x)\bigr)\,\dd y\,\dd \theta\,\dd s}\\
	&\overset{\eqref{eq:hoelder_semi_norm}, \eqref{eq:to_show_heatkernel_3}}\lesssim & [f]_{\alpha}{x}\int_{-\infty}^{-x^{2}}\frac{1}{(-s)^{1+\frac{1}{2}-\alpha}}\,\dd s
	\lesssim [f]_{\alpha}{x}^{2\alpha},
\end{eqnarray*}
where for the second equality we have used
$\int H_{tx}(\cdot,y)\,\dd y=0$ (since $\int H_x(\cdot,y)\,\dd y=0)$.
For \eqref{1_a_gamma}, we use instead
\begin{align*}
	\MoveEqLeft
	(u_{t}-f)(0,x)-(u_{t}-f)(0,0)\\
	&=\int_{-x^{4}}^{0}\int H_{t}(-s,x-y)\bigl(f(s,y)-f(s,x)\bigr)\,\dd y\,\dd s-\int_{-x^{4}}^{0}\int H_{t}(-s,-y)\bigl(f(s,y)-f(s,0)\bigr)\,\dd y\,\dd s\\
	&\qquad +x\int_{-\infty}^{-x^{4}}\int_{0}^{1}\int H_{tx}(-s,\theta x-y)\left(f(s,y)-f(s,\theta x)\right)\,\dd y\,\dd \theta\,\dd s
\end{align*}
and argue as in the proof of \eqref{1_a_beta}.

Next, we will show the second estimate in \eqref{eq:Schauder_est_whole_space}.
Replacing $u_{xx}$ by $v$ and $u_t-f$ by $g$, we see that it is enough to show for
\begin{align}\label{eq:simplified_eq}
	-v+v_{xx}=g
\end{align}
that
$[v]_{\alpha}+[v_{xx}]_{\alpha}\lesssim [g]_{\alpha}.
$
Using the equation, we see that it is enough to show
\begin{align}\label{eq:1_b_alpha-beta}
	[v]_{\alpha}\lesssim [g]_{\alpha}.
\end{align}
We recall the representation formula
\begin{align*}
	v(t,x)=\int G(x-y)g(t,y)\,\dd y\qquad\text{for}\quad
	G(x):=\frac{1}{2}\exp(-\abs{x}).
\end{align*}
Note that $\int G\,\dd x=1$ implies $\norm{v} \leq \norm{g} $. By linearity and translation invariance, we thus obtain
\begin{align*}
	\norm{v(t,\cdot+z)-v(t,\cdot)}&\lesssim \norm{g(t,\cdot+z)-g(t,\cdot)}\leq [g]_{\alpha}\left( z^4 \wedge z^2 \right)^{\alpha},\\
	\norm{v(\cdot+s,x)-v(\cdot,x)}&\lesssim \norm{g(\cdot+s,x)-g(\cdot,x)}\leq [g]_{\alpha} \abs{s}^{\alpha},
\end{align*}
and hence \eqref{eq:1_b_alpha-beta}.

\step{Step 2: Poisson kernel}

For a solution of
\begin{align}\label{eq:fourth_order_u}
	\left\{
	\begin{aligned}
		u_{t}-u_{xx}+u_{xxxx}&=0		&&\text{on }	&&t\in \R,	&&x\in (-\infty,0),\\
		u&=v							&&\text{for }	&&t\in \R,	&& x=0,\\
		u_{xx}&=0						&&\text{for }	&&t\in \R,	&& x=0,
	\end{aligned}
	\right.
\end{align}
we will show the representation formula
\begin{align}\label{eq:representation2}
	u(t,y)=\int P(t-s,y)v(s)\,\dd s
\end{align}
and for $\alpha\in [0,\frac{1}{4})$ the estimate
\begin{align}\label{eq:estimate_poisson}
	\int \bigl(\abs{P(t,x)}+x\abs{P_{x}(t,x)}\bigr) t^{\alpha}\,\dd t\lesssim \left({x}^{4}\wedge{x}^{2}\right)^{\alpha}.
\end{align}
Here, $P$ is the Poisson kernel (with the understanding that it vanishes for $t\leq 0$):
\begin{align}
	P(t,y):=\left(G_{x}-G_{xxx}\right)(t,0,y)=-2\left(H_{x}-H_{xxx}\right)(t,y),\label{psquare}
\end{align}
where $G$ is the Green's function
\begin{align*}
	G(t,x,y):=H(t,x-y)-H(t,x+y).
\end{align*}

We start with the representation formula \eqref{eq:representation2}. Note that by integration by parts, we have that
\begin{align*}
	\MoveEqLeft
	\int\int_{0}^{\infty}(-w_{t}-w_{xx}+w_{xxxx})u\,\dd x\,\dd t
	=\int\int_{0}^{\infty}w(u_{t}-u_{xx}+u_{xxxx})\,\dd x\,\dd t\\
	&\quad -\int (uw_{x}-u_{x}w-uw_{xxx}+u_{x}w_{xx}-u_{xx}w_{x}+u_{xxx}w)\Big|_{x=0}\,\dd t\\
	&\overset{\eqref{eq:fourth_order_u}}=\int \bigl(v(w_{x}-w_{xxx})-u_{x}(w-w_{xx})+u_{xxx}w\bigr)\Big|_{x=0}\,\dd t.
\end{align*}
Applying this to $w(t,x)=G(s-t,x,y)$, so that
\begin{align*}
	\left\{
	\begin{aligned}
		-w_{t}-w_{xx}+w_{xxxx}&=\delta(t-s)\delta(x-y)		&&\text{on }&&t\in\R,&&x\in (0,\infty),\\
		w&=w_{xx}=0												&&\text{for }&&t\in\R, && x=0,
	\end{aligned}
	\right.
\end{align*}
gives \eqref{eq:representation2} in the form
\begin{align*}
	u(s,y)=\int v(w_{x}-w_{xxx})\Big|_{x=0}\,\dd t=\int P(s-t,y)v(t)\,\dd t.
\end{align*}

We now turn to \eqref{eq:estimate_poisson}, which we split into two parts:
\begin{align}
 \int_0^1 \bigl(\abs{P(t,x)}+x\abs{P_{x}(t,x)}\bigr) t^{\alpha}\,\dd t\lesssim \left({x}^{4}\wedge{x}^{2}\right)^{\alpha}, \label{st1}\\
\int_1^\infty \bigl(\abs{P(t,x)}+x\abs{P_{x}(t,x)}\bigr) t^{\alpha}\,\dd t\lesssim \left({x}^{4}\wedge{x}^{2}\right)^{\alpha}. \label{st2}
\end{align}
As in Step~0, we will establish these estimates by using that $H\approx H^{(4)}$ for $t\ll 1$ and $H\approx H^{(2)}$ for $t\gg1$. This time we need pointwise estimates, however.
Differentiating \eqref{eq:0d} and arguing similarly as in Step~0, we obtain
\begin{align*}
	\abs*{\partial_{x}^{k}(H-H^{(4)})}\lesssim \left(t^{\frac{1}{4}}\right)^{-k+1}\;\text{for}\quad t\leq 1\quad\text{and}\quad
	\abs*{x\partial_{x}^{k}(H-H^{(4)})}\lesssim \left(t^{\frac{1}{4}}\right)^{-k+2}\;\text{for}\quad t\leq 1.
\end{align*}
Since by symmetry
\begin{align*}
	\partial_{x}^{k}(H-H^{(4)})=0\;\text{for }x=0,\,k\text{ odd}\quad\text{and}\quad
	x\partial_{x}^{k}(H-H^{(4)})=0\;\text{for }x=0,\,\text{ all }k,
\end{align*}
we have the estimates
%
\begin{align*}\begin{aligned}
&	\abs*{(H_{x}-H_{xxx})-(H^{(4)}_x-H^{(4)}_{xxx})}\\
	&\qquad\lesssim \abs*{(H-H^{(4)})_{x}}+\abs*{(H-H^{(4)})_{xxx}}\\
	&\qquad\lesssim \left(t^{\frac{1}{4}}\right)^{-1+1}+\left(t^{\frac{1}{4}}\right)^{-3+1}
	\lesssim \frac{1}{t^{\frac{1}{2}}},\\
&	\abs*{(H_{x}-H_{xxx})-(H^{(4)}_x-H^{(4)}_{xxx})}\\
	&\qquad\lesssim {x}\left(\norm{(H-H^{(4)})_{xx}}+\norm{(H-H^{(4)})_{xxxx}}\right)\\
	&\qquad\lesssim {x}\left(\left(t^{\frac{1}{4}}\right)^{-2+1}+\left(t^{\frac{1}{4}}\right)^{-4+1}\right)
	\lesssim {x}\frac{1}{t^{\frac{3}{4}}},\\
&	\abs*{x(H_{x}-H_{xxx})_{x}-x(H^{(4)}_x-H^{(4)}_{xxx})_{x}}\\
	&\qquad\lesssim \abs*{x(H-H^{(4)})_{xx}}+\abs*{x(H-H^{(4)})_{xxxx}}\\
	&\qquad\lesssim \left(t^{\frac{1}{4}}\right)^{-1+1}+\left(t^{\frac{1}{4}}\right)^{-3+1}
	\lesssim \frac{1}{t^{\frac{1}{2}}},\\
&	\abs*{x(H_{x}-H_{xxx})_{x}-x(H^{(4)}_x-H^{(4)}_{xxx})_{x}}\\
	&\qquad\lesssim {x}\sup\left( \abs*{\bigl(x(H-H^{(4)})_{xx}\bigr)_{x}}+\abs*{\bigl(x(H-H^{(4)})_{xxxx}\bigr)_{x}}\right)\\
	&\qquad\lesssim {x}\left(\left(t^{\frac{1}{4}}\right)^{-2+1}+\left(t^{\frac{1}{4}}\right)^{-3+2}+
\left(t^{\frac{1}{4}}\right)^{-4+1}+\left(t^{\frac{1}{4}}\right)^{-5+2}\right)\lesssim {x}\frac{1}{t^{\frac{3}{4}}}.
\end{aligned}\end{align*}
Combining these estimates yields
\begin{align*}	 \abs*{(H_{x}-H_{xxx})-(H^{(4)}_x-H^{(4)}_{xxx})}+\abs*{x(H_{x}-H_{xxx})_{x}-x(H^{(4)}_x-H^{(4)}_{xxx})_{x}}\lesssim \frac{{x}}{t^{\frac{3}{4}}}\wedge\frac{1}{t^{\frac{1}{2}}}
\end{align*}
and implies
\begin{align*}
& \int_{0}^{1}\left(\abs*{(H_{x}-H_{xxx})-(H^{(4)}_x-H^{(4)}_{xxx})}
+\abs*{x(H_{x}-H_{xxx})_{x}-x(H^{(4)}_x-H^{(4)}_{xxx})_{x}}\right)t^{\alpha}\,\dd t\\
	&\lesssim {x}\int_{0}^{1}t^{\alpha-\frac{3}{4}}\wedge \int_{0}^{1}t^{\alpha-\frac{1}{2}}\,\dd t
	\lesssim {x}\wedge 1
	\overset{\alpha\leq \frac{1}{4}}\lesssim \left({x}^{2}\wedge{x}^{4}\right)^{\alpha}.
\end{align*}
According to the representation \eqref{psquare} and the triangle inequality, it suffices for \eqref{st1} to show
\begin{align}\label{eq:toshowtriangle}
	 \int_{0}^{1}\left(\abs{H^{(4)}_x}+\abs{xH^{(4)}_{xx}}+\abs{H^{(4)}_{xxx}}+\abs{xH^{(4)}_{xxxx}}\right)t^{\alpha}\,\dd t\lesssim \left({x}^{2}\wedge{x}^{4}\right)^{\alpha}.
\end{align}
For $\hat x=\frac{x}{t^{\frac{1}{4}}}$ and $t\leq 1$ we recall \eqref{B2} and estimate
\begin{align*}
	\MoveEqLeft
	\abs{H^{(4)}_x}+\abs{xH^{(4)}_{xx}}+\abs{H^{(4)}_{xxx}}+\abs{xH^{(4)}_{xxxx}}\\
	&= \frac{1}{t^{\frac{1}{2}}}\left(\abs*{H^{(4)}_{\hat x}\left(1,\hat x\right)}
	+\abs*{\hat x}\abs*{H^{(4)}_{\hat x\hat x}\left(1,\hat x\right)}\right)
	+\frac{1}{t}\left(\abs*{H^{(4)}_{\hat x\hat x\hat x}\left(1,\hat x\right)}
	+\abs*{\hat x}\abs*{H^{(4)}_{\hat x\hat x\hat x\hat x}\left(1,\hat x\right)}\right)\\
	&\leq \frac{1}{t}\left(\abs{H^{(4)}_{\hat x}}+\abs{\hat x}\abs{H^{(4)}_{\hat x\hat x}}+\abs{\hat x}\abs{H^{(4)}_{\hat x\hat x\hat x}}+\abs{H^{(4)}_{\hat x\hat x\hat x\hat x}}\right)\left(1,\hat x\right).
\end{align*}
Using this in \eqref{eq:toshowtriangle} and changing variables $t\to \hat x=\frac{x}{t^{\frac{1}{4}}}$, we obtain
\begin{eqnarray*}
	\lefteqn{ \int_{0}^{1}\left(\abs{H^{(4)}_x}+\abs{xH^{(4)}_{xx}}+\abs{H^{(4)}_{xxx}}+\abs{xH^{(4)}_{xxxx}}\right)t^{\alpha}\,\dd t}\\
	&\lesssim& \int_{x}^{\infty}\abs*{\frac{x}{\hat x}}^{4\alpha}\left(\abs{H^{(4)}_{\hat x}}+\abs{\hat x}\abs{H^{(4)}_{\hat x\hat x}}+\abs{H^{(4)}_{\hat x\hat x\hat x}}+\abs{\hat x}\abs{H^{(4)}_{\hat x\hat x\hat x\hat x}}\right)
	\frac{\dd \hat x}{\abs{\hat x}}\\
	&\leq&
	\begin{cases}
		x^{4\alpha}\int_{0}^{\infty}\frac{1}{\abs{\hat x}^{4\alpha}}\left(\abs{H^{(4)}_{\hat x}}+\abs{\hat x}\abs{H^{(4)}_{\hat x\hat x}}+\abs{H^{(4)}_{\hat x\hat x\hat x}}+\abs{\hat x}\abs{H^{(4)}_{\hat x\hat x\hat x\hat x}}\right)
	\frac{\dd \hat x}{\abs{\hat x}} & x\leq 1\\
		\int_{1}^{\infty}\left(\abs{H^{(4)}_{\hat x}}+\abs{\hat x}\abs{H^{(4)}_{\hat x\hat x}}+\abs{H^{(4)}_{\hat x\hat x\hat x}}+\abs{\hat x}\abs{H^{(4)}_{\hat x\hat x\hat x\hat x}}\right)
	\,\dd \hat x & x\geq 1
	\end{cases}\\
	&\lesssim& {x}^{4\alpha}\wedge 1
	\overset{\alpha\geq 0}\lesssim \left({x}^{4}\wedge {x}^{2}\right)^{\alpha}.
\end{eqnarray*}
Here we have used that $\alpha\in (0,\frac{1}{4})$ and properties of $\abs{H^{(4)}_{\hat x}}+\abs{\hat x}\abs{H^{(4)}_{\hat x\hat x}}+
\abs{H^{(4)}_{\hat x\hat x\hat x}}+\abs{\hat x}\abs{H^{(4)}_{\hat x\hat x\hat x\hat x}}$, namely, that
it is integrable at infinity and, by evenness of $H^{(4)}$, vanishes linearly at zero.
This establishes \eqref{st1}.

For \eqref{st2}, we use a similar argument; we omit the details and summarize the main points. In this case the pointwise estimates for $t\geq 1$ are
\begin{align*}
	\abs*{\partial_{x}^{k}(H-H^{(2)})}\lesssim \left(t^{\frac{1}{2}}\right)^{-k-2}\quad\text{and}\quad
	\abs*{x\partial_{x}^{k}(H-H^{(2)})}\lesssim \left(t^{\frac{1}{2}}\right)^{-k-1}.
\end{align*}
We use these estimates to deduce
\begin{align*}	 \abs*{(H_{x}-H_{xxx})-(H^{(2)}_x-H^{(2)}_{xxx})}+\abs*{x(H_{x}-H_{xxx})_{x}-x\left(H^{(2)}_x-H^{(2)}_{xxx}\right)_{x}}\lesssim \frac{{x}}{t^2}\wedge\frac{1}{t^{-\frac{3}{2}}}.
\end{align*}
From here we deduce the integral bound
\begin{align*}
\int_1^\infty\left(\abs*{(H_{x}-H_{xxx})-(H^{(2)}_x-H^{(2)}_{xxx})}\right.
&\left.+\abs*{x(H_{x}-H_{xxx})_{x}-x(H^{(2)}_x-H^{(2)}_{xxx})_{x}}\right)t^{\alpha}\,\dd t\\
	&\lesssim {x}\wedge 1
	\overset{\alpha\leq \frac{1}{4}}\lesssim \left({x}^{2}\wedge{x}^{4}\right)^{\alpha}.
\end{align*}
Hence it suffices to consider the heat kernel, for which we observe
\begin{align*}	 \int_{1}^{\infty}\left(\abs{H^{(2)}_x}+\abs{xH^{(2)}_{xx}}+\abs{H^{(2)}_{xxx}}+\abs{xH^{(2)}_{xxxx}}\right)
t^{\alpha}\,\dd t\lesssim
{x}\wedge{x}^{2\alpha}\overset{\alpha\leq \frac{1}{4}}\lesssim \left(x^4\wedge x^2\right)^\alpha.
\end{align*}

\step{Step 3: Dirichlet problem without forcing on the half-line: Estimates for case (ii)}

Here we will show that a solution $u(t,x)$ of \eqref{eq:fourth_order_u} satisfies
\begin{align*}
	[u_{t}]_{\alpha}+[u_{xx}]_{\alpha}+[u_{xxxx}]_{\alpha}\lesssim [v_{t}]_{\alpha},
\end{align*}
which splits into
\begin{align}\label{eq:Schauder_est_D_half_space}
	[u_{t}]_{\alpha}\lesssim [v_{t}]_{\alpha} \qquad\text{and}\qquad
	[u_{xx}]_{\alpha}+[u_{xxxx}]_{\alpha}\lesssim [u_{t}]_{\alpha}.
\end{align}
We claim that, for the first item in \eqref{eq:Schauder_est_D_half_space}, it is enough to show
\begin{align}
	\abs{u_{t}(t,x)-u_{t}(0,x)}&\lesssim t^{\alpha}[v_{t}]_{\alpha}\qquad &&\text{for } t\geq 0,\label{3_a_alpha}\\
	\abs{u_{t}(0,x)-u_{t}(0,0)}&\lesssim \left({x}^{4}\wedge{x}^{2}\right)^{\alpha}[v_{t}]_{\alpha}\qquad&&\text{for }x\geq 0,\label{3_a_beta}\\
	{x}\abs{u_{tx}(0,x)}&\lesssim \left({x}^{4}\wedge{x}^{2}\right)^{\alpha}[v_{t}]_{\alpha}\qquad&&\text{for }x\geq 0.\label{3_a_gamma}
\end{align}
Indeed, \eqref{3_a_beta} and \eqref{3_a_gamma} yield spatial continuity
\begin{align*}
	\abs{u_{t}(0,x+z)-u_{t}(0,x)}\lesssim \left( z^4 \wedge z^2 \right)^{\alpha}[v_{t}]_{\alpha}\quad\text{for $z\geq 0$.}
\end{align*}
To see this, we distinguish between two cases. For $z\leq x$, we use
\begin{eqnarray*}
	\lefteqn{
	\abs{u_{t}(0,x+z)-u_{t}(0,x)}\leq \int_{x}^{x+z}\abs{u_{tx}(0,y)}\,\dd y}\\
	&\overset{\eqref{3_a_beta}}\lesssim& [v_{t}]_{\alpha}\int_{x}^{x+z}\frac{1}{\abs{y}}\left( y^4 \wedge y^2 \right)^{\alpha}\,\dd y
	\lesssim \frac{\abs{z}}{\abs{x}}\left(x^4\wedge x^2\right)^\alpha\\
& \lesssim &\left( z^4 \wedge z^2 \right)^{\alpha}[v_{t}]_{\alpha},
\end{eqnarray*}
where we have (twice) used that the function $\frac{1}{y}\left(y^4\wedge y^2\right)^\alpha$ is monotone decreasing (since $\alpha\leq \frac{1}{4}$) and that $x\geq z$.
For the case $z\geq x$ we have
\begin{align*}
	\MoveEqLeft
	\abs{u_{t}(0,x+z)-u_{t}(0,x)}\leq \abs{u_{t}(0,x+z)-u_{t}(0,0)}+\abs{u_{t}(0,0)-u_{t}(0,x)}\\
	&\overset{\eqref{3_a_gamma}}\lesssim \left((x+z)^4\wedge(x+z)^2 \right)^{\alpha}[v_{t}]_{\alpha}
	\overset{x\leq z}\lesssim\left( z^4 \wedge z^2 \right)^{\alpha}[v_{t}]_{\alpha}.
\end{align*}

We now turn to the arguments for \eqref{3_a_alpha}--\eqref{3_a_gamma}. From the representation \eqref{eq:representation2} and differentiation with respect to time, we obtain
\begin{align}\label{eq:representation_ut}
	u_{t}(t,x)
=\int P(s,x)v_{t}(t-s)\,\dd s,
\end{align}
from which we deduce \eqref{3_a_alpha}:
\begin{align*}
	\abs{u_{t}(t,x)-u_{t}(0,x)}=\abs*{\int P(s,x)\bigl(v_{t}(t-s)-v_{t}(-s)\bigr)\,\dd s}
	\overset{\eqref{eq:estimate_poisson}\text{ for $\alpha=0$}}\lesssim t^{\alpha}[v_t]_{\alpha}.
\end{align*}

For \eqref{3_a_beta}, we note that from the boundary condition in \eqref{eq:fourth_order_u}, $\int P(s,x)\,\dd s=1$ (which can be seen from the representation formula \eqref{eq:representation2} with $u\equiv 1$),
and the representation \eqref{eq:representation_ut}, we obtain
\begin{align*}
	u_{t}(0,x)-u_{t}(0,0)=\int P(s,x)\bigl(v_{t}(-s)-v_{t}(0)\bigr)\,\dd s,
\end{align*}
so that as desired
\begin{align*}
	\abs{u_{t}(0,x)-u_{t}(0,0)}\leq [v_t]_{\alpha}\int P(s,x)s^{\alpha}\,\dd s
	\overset{\eqref{eq:estimate_poisson}}\lesssim [v_t]_{\alpha}\left({x}^{4}\wedge{x}^{2}\right)^{\alpha}.
\end{align*}
To show \eqref{3_a_gamma}, we first note that \eqref{eq:representation_ut} and $\int P_{x}(s,x)\,\dd s=0$ imply
\begin{align*}
	u_{tx}(0,x)=\int P_{x}(s,x)v_{t}(-s)\,\dd s=\int P_{x}(s,x)\bigl(v_{t}(-s)-v_{t}(0)\bigr)\,\dd s,
\end{align*}
so that
\begin{align*}
	{x}\abs{u_{tx}(0,x)}\leq [v_t]_{\alpha}\int {x} \abs{P_{x}(s,x)}s^{\alpha}\,\dd s
	\overset{\eqref{eq:estimate_poisson}}\lesssim [v_t]_{\alpha}\left({x}^{4}\wedge {x}^{2}\right)^{\alpha}.
\end{align*}

It remains to show the second estimate in \eqref{eq:Schauder_est_D_half_space}. To this end, we consider $w:=u_{xx}$, $g:=u_t$ related by
\begin{align*}\begin{aligned}
  w-w_{xx}&=g &&x>0\\
  w&=0&&x=0.\end{aligned}
\end{align*}
Since $t$ is just a parameter (and using the equation), it suffices to show
\begin{align*}
  \norm{w}\lesssim\norm{g}\quad\text{and}\quad [w]_\alpha \lesssim [g]_\alpha,
\end{align*}
where both denote purely spatial norms. For the first estimate, we proceed as in Step 1. We use odd reflection of $w$, $g$ and the representation
\begin{align*}
  w(x)=\int G(x-y)\,g(y)\dd y ,\quad\text{with}\quad G(x)=\fr12 \exp(-\abs{x})
\end{align*}
denoting the fundamental solution of $1-\partial_x^2$.
For the second estimate, we use the Green's function
\begin{align*}
  G(x,y):=G(x-y)-G(x+y)
\end{align*}
xand the representation $w(x)=\int_0^\infty G(x,y)g(y)\,\dd y.$

\end{appendix}
\section*{Acknowledgements}
S. Scholtes was
partially supported by DFG Grant WE 5760/1-1.

\bibliographystyle{amsalpha}

\bibliography{references}{}

%
%
%
%
%
%
%
%
%
%

\end{document}